\documentclass[EJP]{ejpecp} % replace ECP by EJP if needed.
\usepackage{mathrsfs}

\newtheorem{defi}{Definition}[section]
\newtheorem{lem}[defi]{Lemma}
\newtheorem{thm}[defi]{Theorem}
\newtheorem{cor}[defi]{Corollary}
\newtheorem{prop}[defi]{Proposition}
\theoremstyle{remark}
\newtheorem{rem}[defi]{Remark}

\newcommand{\N}{\mathbb{N}}

\newcommand{\E}{\mathbb{E}}
\renewcommand{\P}{\mathbb{P}}
\newcommand{\R}{\mathbb{R}}
\newcommand{\eps}{\varepsilon}

\SHORTTITLE{Click times for a variant of Muller's ratchet}

\TITLE{Quasi-equilibria and click times for a variant of Muller's ratchet} % \thanks is optional. Insert line breaks with \\

\AUTHORS{%
  Adri\'an~Gonz\'alez~Casanova\footnote{Universidad Nacional Aut\'onoma de M\' exico (UNAM), Instituto de Matem\'aticas, Circuito exterior, Ciudad Universitaria, 04510, M\'exico, 11 and University of California, Berkeley.
Department of Statistics, 
367 Evans Hall, University of California,
Berkeley, CA 94720-386.
    \EMAIL{adriangcs@matem.unam.mx, gonzalez.casanova@berkeley.edu}}
  \and %% remove this line and below if single author
  Charline~Smadi\footnote{Univ. Grenoble Alpes, INRAE, LESSEM, 38000 Grenoble, France
 and Univ. Grenoble Alpes, CNRS, Institut Fourier, 38000 Grenoble, France. \BEMAIL{charline.smadi@inrae.fr} }
   \and %% remove this line and below if single author
  Anton~Wakolbinger\footnote{Goethe-Universit\"at, Institut f\"ur Mathematik, 60629 Frankfurt am Main, Germany. \BEMAIL{wakolbinger@math.uni-frankfurt.de} }}%AUTHORS
\KEYWORDS{Muller's ratchet; click rate; quasi-equilibrium; tournament selection; Harris graphical representation; duality; ancestral selection graph; first passage percolation} % Separate items with ;

\AMSSUBJ{Primary 60K35; secondary 92D15} % Edit. Separate items with ;
\VOLUME{0}
\YEAR{2020}
\PAPERNUM{0}
\DOI{10.1214/YY-TN}

\ABSTRACT{Consider a population of $N$ individuals, each of them carrying a type in $\mathbb N_0$. The population evolves according to a Moran dynamics with selection and mutation, where an individual of type $k$ has the same selective advantage over all individuals with type $k' > k$, and type $k$ mutates to type $k+1$ at a constant rate. This model is thus a variation of the classical Muller's ratchet: there the selective advantage is proportional to $k'-k$. For a regime of selection strength and mutation rates which is between the regimes of weak and strong selection/mutation, we obtain the asymptotic rate of the {\em click times} of the ratchet (i.e. the times at which the hitherto minimal (`best') type in the population is lost), and reveal the quasi-stationary type frequency profile between clicks. The large population limit of this profile is characterized as the normalized attractor of a ``dual'' hierarchical multitype logistic system, and also via the distribution of the final minimal displacement in a branching random walk with one-sided steps. An important role in the proofs is played by a graphical representation of the model, both forward and backward in time, and a central tool is the ancestral selection graph decorated by mutations.}

\begin{document}

%\tableofcontents

\section{Introduction}

A well-known model in population genetics named {\em Muller's ratchet}  (cf. \cite{M64, MS78, EPW09, PSW12, ME13} and references therein) considers, in its bare bones version, the interplay between selection and stepwise (slightly) deleterious mutation in the absence of recombination, for a population of constant size $N \gg 1$. Briefly spoken, selection decreases  the expected reproductive success of those individuals that have a higher mutational load. Since the mutational load is inherited, selection tends to decrease the overall mutational load in the population. Because of the assumed unidirectional mutation and randomness in the reproduction, the currently lightest load eventually disappears from the population. This is phrased by H.J. Muller in his pioneering paper \cite{M64} as follows:  {\em ``\ldots an irreversible ratchet mechanism exists  in the non-recombining species \ldots that prevents selection, even if intensified, from reducing  the mutational loads below the lightest  that were in existence when the intensified selection started, whereas, contrariwise, `drift',  and what might be called  `selective noise' must allow occasional slips  of the lightest loads  in the direction of increased weight.''}
Thanks to recombination however, sexual organisms are able to avoid such an accumulation of deleterious mutations. This could be one of the explanations of the ubiquity of sexual reproduction among eukaryotes despite its many costs \cite{felsenstein1974evolutionary}.

We will now give a short description of the model and the main questions addressed in this paper. A more detailed presentation of the model, a formulation of the main results and a preview of the proof strategy will be given in Section \ref{MMR}.

Each individual carries, as its current  {\em type}, a number $\kappa$ 
of deleterious mutations. The number of mutations along each lineage increases by~$1$ at a rate $m_N$, and as soon as an individual reproduces, its current type is inherited to its 'daughter'. Reproduction happens according to a Moran dynamics with selection, where the fitness difference between two individuals of type $\kappa$ and $\kappa'$  is $\frac {s_N}N \Phi(\kappa'-\kappa)$ for a selection 
parameter $s_N$ and a non-decreasing antisymmetric function $\Phi: \mathbb Z \to \mathbb R$.  (For the classical variant of Muller's ratchet, $\Phi$ is 
the identity function  on $\mathbb Z$, which means that the effect of selection is proportional to the difference of mutational loads.) The individual-based dynamics in this model, which we briefly call the {\em $\Phi$-ratchet}, arises as an independent superposition of the following three ingredients:

\smallskip  \noindent
 {\bf Moran resampling.}  For each pair of individuals $(i,j)$, irrespective of their types,  $j$ is replaced by a newborn daughter of individual $i$  at rate $\frac 1{2N}$.

 \smallskip\noindent
{\bf Selective reproduction.} For each pair of individuals $(i,j)$ of types $\kappa$ and~$\kappa'$,  individual~$j$ is replaced by a newborn daughter of individual $i$ at rate $\frac {s_N}N \Phi(\kappa'-\kappa)\mathbf 1_{\{\kappa' > \kappa\}}$.

\smallskip\noindent
{\bf Stepwise mutation.}  For each individual, its type is increased by 1 at rate $m_N$.

\smallskip
 This dynamics leads to a sequence of times at which the currently lowest (and thus selectively 'best') type in the population disappears. These times will be referred to as 'click times of the ratchet'. In certain regimes of the parameters $s_N$ and~$m_N$,  as $N$ becomes large the click times happen rarely and a quasi-stationary type frequency profile builds up between the click times.
The following questions thus call for an answer:
\begin{enumerate}
\item[A.] What is the rate (of the click times) of the ratchet?
\item[B.] What is the quasi-stationary type frequency profile?
\end{enumerate}
For the classical variant of Muller's ratchet, a fully rigorous asymptotic analysis of these problems  beyond existence results is a notorioulsy difficult and in large parts still unsolved task, see e.g. \cite{MPV22}.
 
We propose a variant of the fitness function $\Phi$ which leads to a model that  turns out to be tractable by modern probabilistic techniques, allowing for quantitative results for the rate of the ratchet and the quasi-stationary type profile. This specific choice, denoted by $\varphi$, is 

\vspace{-0.4cm}
\begin{equation}\label{defphi}
\varphi(\kappa'-\kappa) := \mathbf 1_{\{\kappa'-\kappa > 0\}} - \mathbf 1_{\{\kappa' - \kappa<0\}}.
\end{equation}
The essential difference to the classical variant of Muller's ratchet thus is that the selective advantage does not depend on the {\em value} of the difference of $\kappa'$ and $\kappa$, but only on the difference's {\em sign}. This corresponds to replacing {\em proportional selection} by  {\em binary tournament selection} \cite{blickle1996comparison, paixao2015toward}, whose  effect may be imagined as due to pairwise {\em fights} between randomly chosen individuals, where the individual of 'better' type outcompetes that of  worse type. From a biological point of view, this provides a clear cut and reasonable alternative to the classical variant. From the perspectives of mathematis, the tradeoff between the complexity and the tractability of the model is attractive. The present work analyses a {\em subcritical} regime in which the mutation-selection ratio $\tfrac {m_N}{s_N}$ remains constant and below $1$. The {\em near-critical} regime, in which $\tfrac {m_N}{s_N} \uparrow 1$, presents additional technical challenges. We have started to study this regime and some interesting links to the classical variant of Muller's ratchet in \cite{igelbrink2023muller}.
\section{Model and main results}\label{MMR}

We now give a  definition of the jump rates of the type frequencies of the $\varphi$-ratchet with selection strength $s_N$ and mutation rate $m_N$. For fixed $N\in \mathbb N = \{1,2, \ldots\}$ and for $\kappa \in \mathbb N_0 = \{0,1,2,\ldots\}$, let $\xi_\kappa(t)= \xi_{\kappa}^{(N)}(t)$ be the proportion (or {\em frequency}) of type $\kappa$-individuals at time $t$. (Here and below we will sometimes suppress the index $N$.)  Denoting by $(e_i, i \in \N_0)$ the canonical basis of $\N_0^{\N_0}$, the process $(\xi_\kappa, \kappa \in \N_0)$ jumps with the following increments:

\begin{itemize}
\item[$\bullet$] Mutation: for $\kappa \in \N_0$,
\begin{equation}\label{mut}
(e_{\kappa+1}-e_{\kappa})/N \quad \text{is added at rate} \quad m_N N \xi_\kappa 
\end{equation}
\item[$\bullet$] Selection: for $\kappa < \kappa'$ 
\begin{equation}\label{sel}
(e_\kappa-e_{\kappa'})/N \quad \text{is added at rate} \quad s_N N \xi_\kappa \xi_{\kappa'} 
\end{equation}
\item[$\bullet$] Resampling: for $\kappa \neq \kappa'$ 
\begin{equation}\label{coal}
(e_\kappa-e_{\kappa'})/N \quad \text{is added at rate} \quad \frac{N}{2} \xi_\kappa \xi_{\kappa'}.
\end{equation}
\end{itemize}
Intuitively spoken, the type $\kappa$-subpopulation is 'fed through mutations' from the type\mbox{$(\kappa-1)$}-subpopulation, is 'selectively attacked' by the  type$(<\kappa)$-subpopulation and  'selectively attacks' the type$(>\kappa)$-subpopulation.
Consequently, these rates imply that for any~$\kappa$, $(\xi_0,...,\xi_\kappa)$ is an autonomous process.
\begin{defi}[Click times and  type frequency profile]\label{clicktimes}
a) The {\em best type at time $t$} is
\begin{equation}\label{besttypexi}
K_N^\ast(t):= \min\left\{\kappa \in \N_0: \xi_{\kappa}^{(N)} (t) > 0\right\}.
\end{equation}
b) The jump times of the counting process $K_N^\ast$  are called the {\em click times} of the ratchet. 
\\
c) The {\em empirical type profile at time $t$ (seen from the currently best type)} is 
\begin{equation}\label{defXk1} X_k^{(N)}(t):= \xi_{K_N^\ast(t)+k}^{(N)}(t), \quad k \in \mathbb N_0. 
\end{equation}
\end{defi}

To obtain quantitative results for the click rates and the quasi-stationary type profile, we will throughout the paper consider the case of {\em moderate} selection and mutation
\begin{equation}\label{moderatesm}
s_N = \frac{\alpha}{f(N)}, \quad m_N= \frac{\mu}{f(N)},
\end{equation}
where $\mu < \alpha$, $f(N)\to \infty$ and $f(N)= o\left(\frac N{\ln \ln N}\right)$ as $N\to \infty$. In particular, this implies that $m_N \to 0$, $Nm_N \to \infty$, and $m_N$  and $s_N$ are of the same order.

\begin{thm}[Asymptotic rate of clicks]\label{thmclicks}
Assume that all individuals at time $0$ are of type~$0$, i.e. $\xi_\kappa^{(N)}(0) = \delta_{0\kappa}$, $\kappa \in \N_0$. Then there exists a sequence $(\theta_N)$ with 
\begin{equation} \label{order_time_ext}
\ln \theta_N \sim 2 (\alpha - \mu + \mu \ln (\mu/\alpha))\frac{N}{f(N)}  \quad \mbox{ as } N \to \infty,
\end{equation}
such that the sequence of time-rescaled click time processes 
$$\left( K^*_N\left(f(N)\theta_N\tau \right)  \right)_{\tau\ge 0}, \quad N=1,2,\ldots$$
converges in distribution as $N \to \infty$ to a rate 1 Poisson counting process (when restricted to $[0,T]$ for each $T>0$.)
\end{thm}
In particular, for the case of  {\em nearly strong selection} 
$ s_N =1/l(N)$, where $l(N)$ is any slowly varying function that converges to infinity with $N$,  Theorem \ref{thmclicks} says that the expected time between clicks is only slightly smaller than exponential in $N$. In contrast to this, for  {\em nearly weak selection}
 with $(\ln\ln N)/N \ll s_N \ll (\ln N)/N$,
 Theorem~\ref{thmclicks} says that  the timescale $f(N)\theta_N$ of clicks is asymptotically only slightly larger than the evolutionary timescale $N$.
\begin{thm}[Quasi-stationary type frequency profile] \label{thmprofile}
\begin{enumerate}
\item[a)]
Assume that all individuals at time $0$ are of type $0$. Let $(t_N)$ be a deterministic sequence of times such that
\begin{equation}\label{condtN}
 \frac{t_N}{ f(N)\ln N} \to \infty \quad \mbox{as } N \to \infty.
 \end{equation}
Then the empirical type frequency profile defined in \eqref{defXk1} obeys for all $k \in \N_0$
 \begin{equation*}
 X_k^{(N)}(t_N) \to p_k \quad \mbox{in probability as } N\to \infty,
 \end{equation*}
 where $(p_k)_{k\in \N_0}$ is a sequence of probability weights given by the recursion 
 \begin{equation}\label{recursion}
p_0 = 1-\frac \mu \alpha \quad and \quad p_{k}^2-p_k\left(1-\frac\mu\alpha-2\sum_{k'=0}^{k-1}p_{k'}\right) = \frac\mu\alpha p_{k-1}, \quad k \geq 1 .
 \end{equation}
 \item [b)] The recursion \eqref{recursion} is equivalent to the (mutation-selection equilibrium) system
\begin{eqnarray}\label{systeq} \alpha \, p_{k}  \left(\sum_{k'\in \mathbb N_0} p_{k'}\left(\mathbf 1_{\{k'>k\}}- \mathbf 1_{\{k'<k\}}\right)\right) =\mu \left(p_{k}-p_{k-1}\right), \quad k \geq 0, 
 \end{eqnarray}
 with the boundary conditions  $p_{-1} = 0$, $p_0 > 0$, $\sum\limits_{k\in \N_0}p_k=1$. 
\item[c)]
Two alternative probabilistic descriptions of $(p_k)_{k \in \N_0}$ given by \eqref{recursion} are as follows:
\begin{itemize}
\item Consider a Yule tree with splitting rate $\alpha$ whose branches are decorated by a rate~$\mu$ Poisson point process. Then $p_k$ is the probability that there is an infinite lineage carrying exactly $k$ points but no infinite lineage with less than $k$ points. 
\item Consider a branching random walk on $\N_0$ starting with one individual at the origin with binary branching at rate $\alpha$ (and no death) and with migration of individuals from $k$ to $k+1$ at rate $\mu$. 
Then, as $t\to \infty$, the minimal position of the individuals alive at time $t$ converges in law to a random variable with distribution $(p_k)_{k \in \N_0}$. 
\end{itemize}
\item[d)] With $\rho:= \mu/\alpha$, the tails of the probability weights $(p_k)$ given by  \eqref{recursion} are represented by the iterations
\begin{equation} 
\sum_{k>\ell} p_k =  \underbrace{\mathfrak{G}\circ... \circ\mathfrak{G}}_{\ell  \text{ times}}(\rho), \qquad \ell \in \N_0,
 \end{equation}
where
\vspace{-0.3cm}
\begin{equation}\label{defG}
\mathfrak{G}(u) :=  \frac 12 \left(1+\rho-\sqrt{(1+\rho)^2-4\rho u}\right), \quad 0\le u \le 1.
\end{equation}
 \item [e)]  Let $(p_k)_{k \in \N_0}$ be the probability weights given by \eqref{recursion}. Then 
\begin{itemize}
\item[(i)]
for $0 < \frac {\mu} {\alpha} < \frac 23 $, \quad \quad $k \mapsto p_k$ is strictly decreasing,
\item[(ii)]
for $\frac {\mu} {\alpha} = \frac 23$, \quad \qquad \quad  $k \mapsto p_k$ is decreasing with $p_0 = p_1 > p_2 > \cdots$, 
\item[(iii)]
for $\frac 23 < \frac {\mu} {\alpha} < 1$,   \quad $(p_k)_{k\in \N_0}$ is unimodal in the sense that there exist $k_1 \le k_2$ with $k_2-k_1\le 1$   for which $p_0<p_1<\cdots<p_{k_1} =p_{ k_2}$ and $p_{k_2} > p_{ k_2+1} > \cdots$.
\end{itemize}
In any case, $p_k \sim C\cdot \left(\frac \mu{\mu+\alpha}\right)^k$ as $k\to \infty$ for some constant  $C$ depending on $\frac \mu \alpha$. 

\end{enumerate}
\end{thm}
\begin{rem} a)  Eq.~\eqref{systeq} characterizes the type frequency profile $(p_k)_{k\in \N_0}$ as the fixed point of a deterministic mutation-selection equilibrium, with the out-flux due to mutation on its right hand side and the in-flux due to selection on its left hand side. The latter can be  written as $p_k \alpha \sum_{k'\in \N_0} \varphi(k'-k)p_{k'}$ with $\varphi$ as in \eqref{defphi}.   If  $\varphi$ would be replaced by $\Phi(k'-k) := k'-k$, then (cf. \cite{H78})  the solution of \eqref{systeq} would be the Poisson weights with parameter $ \mu / \alpha$. 
\\
b) An essential advantage of the form \eqref{defphi} of the fitness function  is that it opens the way to a mathematically tractable analysis of the probabilistic system via a dual process within a graphical representation. The latter will be developed in Sections \ref{GRep} and \ref{PAP}, and the dual process will appear as a hierarchy of competing   logistic populations in Sections \ref{PercASG} and \ref{processZ}.  As an appetizer and short preview, let us point to the duality between the {\em size of the initially best class,} 
$$Y_0^{(N)} := N\xi_0^{(N)}$$ 
and the process $Z_0^{(N)}$ whose jump rates are given by \eqref{Z0rates}. According to \eqref{mut}-\eqref{coal}, the process $Y_0^{(N)}$ has the jump rates
\begin{eqnarray}\begin{split}\label{Y0rates}
&n \to n+1 \quad \text{ at rate} \quad n\left(\frac 12+s\right)\left(1-\frac {n}N\right), \\
&n \to n-1 \quad \text{at rate} \quad n\left(\frac 12\left(1-\frac {n}N\right)+m_N \right).
\end{split}
\end{eqnarray}
An inspection of \eqref{Y0rates} and \eqref{Z0rates} shows that $Y_0^{(N)}$ and $Z_0^{(N)}$ are in {\em hypergeometric duality}  \cite[Example 4.7]{Jansen2014duality}. In more probabilistic terms, $Y_0^{(N)}$ is a two-type Moran counting process with directional selection and one-way mutation, and $Z_0^{(N)}$ counts the number of potential ancestors in a {\em decorated } ancestral selection graph whose branches split at rate $s_N$, coalesce at rate $\frac 1N$ per pair, and are pruned at rate $m_N$ (\cite{KN97, baake2018probabilistic}). Indeed, the rate figuring in Theorem~\ref{thmclicks} appears also in Lemma~\ref{exp_rate} as the asymptotic rate of extinction of $Z_0^{(N)}$.
This constitutes the ground level ($k=0$) in a duality between the type frequency process $(N\xi_k^{(N)})_{k\in \N}$ and a ``hierarchy of logistic competitions'' $(Z_k^{(N)})_{k\in \N}$. The latter will be introduced in  Section \ref{processZ}, and the duality  will appear in the ``grand graphical picture'' developed in Sections \ref{GRep}-\ref{PercASG}.  %We will describe this in more detail in the following Remark.} 
\end{rem}

We end this section by a description of the proof strategy of Theorems \ref{thmclicks} and \ref{thmprofile}, with some of  the ideas illustrated by Figure \ref{fig:boat4}.

\begin{figure}[h!]
\begin{center}
  \includegraphics[width=.5\linewidth]{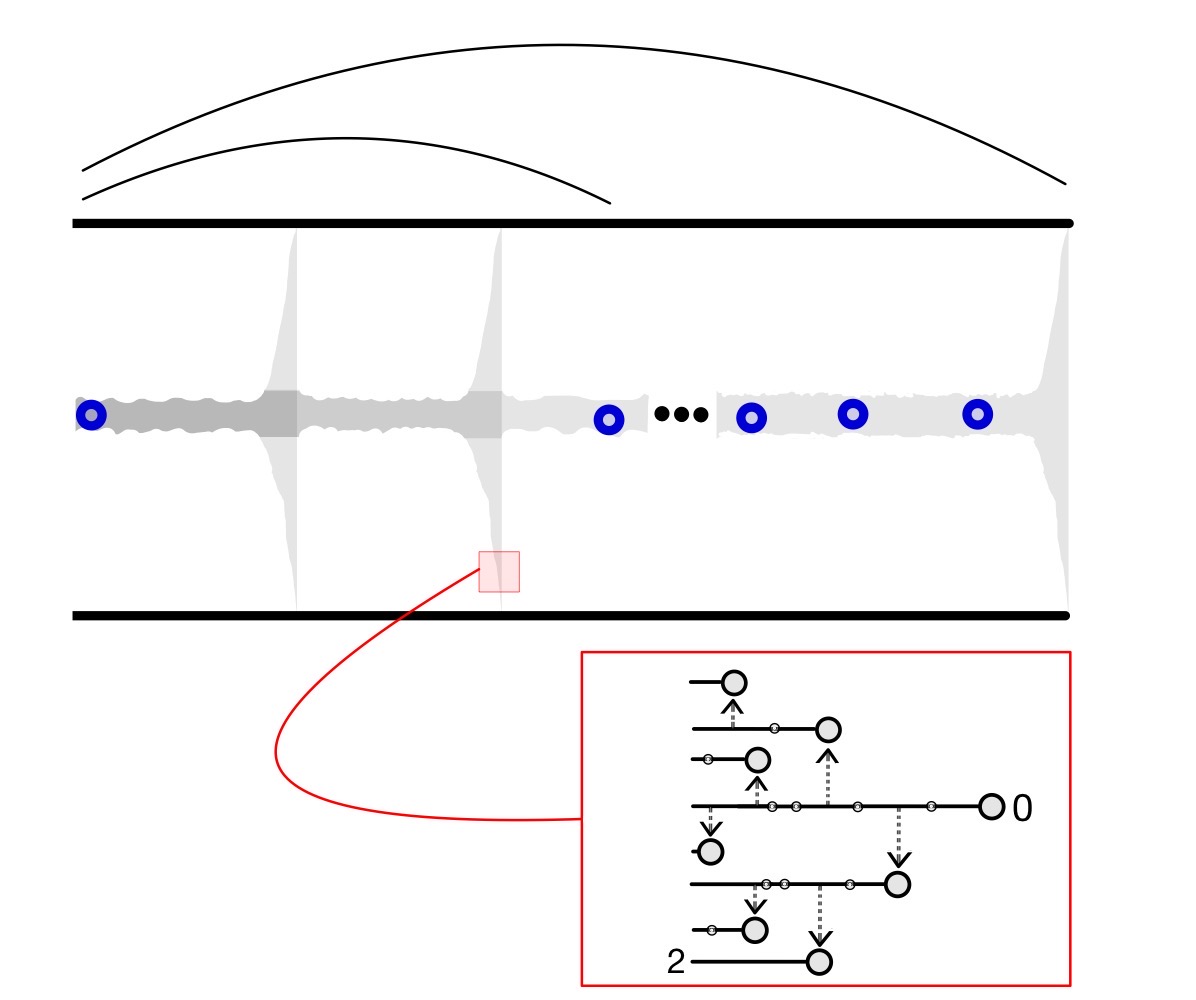}
  \end{center}
  \caption{This cartoon reflects the strategy for proving  Theorems~\ref{thmclicks} and~\ref{thmprofile}. The time arrow for the evolution of the Ancestral Selection Graph points from right to left. The ``minimal load ASG'' is symbolized by the grey band amidst the total population that is confined by the two solid horizontal lines.
   ASGs started from the entire population merge on the scale $f(N)\ln N$  into  the minimum load ASG. 
  Along the latter, clicks (depicted by thick circles) happen asymptotically at (the much slower) scale $\theta_N f(N)$, and these can be coupled locally on that scale with the clicks forward in time.
  Most of the variability on the number of mutations among contemporaneous individuals comes from mutations acquired in their recent past. These recent mutations can be studied via duality by means of a Yule process approximation on the scale $f(N)$, leading to the quasi-stationary type frequency profile $(p_k)_{k\in \N_0}$.}
  \label{fig:boat4}
\end{figure}

Think of some late time $u=u_N$ being fixed, located at the far right of Figure \ref{fig:boat4}, and denote the population living at this time by $\mathfrak P_u =: \mathfrak P$. For $t < u$, the Ancestral Selection Graph (ASG) $\mathscr A_t= \mathscr A_t^{\mathfrak P}$ consists of all the {\em potential ancestors} of~$\mathfrak P$ that live at time $t$, see  Definition~\ref{ASGdefi}.  Each individual in $\mathscr A_t$  has a certain ``mutational distance'' (\mbox{$\mathcal M$-distance}, see Definition \ref{defi43}) from $\mathfrak P$. Those
 individuals in $\mathscr A_t$ that  have the smallest $\mathcal M$-distance from $\mathfrak P$ among their contemporaneans, make up the {\em minumum load ASG} $\bar {\mathscr A}_t$, see~Definition \ref{colouredASG}.

\begin{itemize}
\item  Every once in a while, as $t$ decreases (i.e. wanders to the left in Figure \ref{fig:boat4})  the $\mathcal M$-distance between $\bar {\mathscr A}_t$ and $\mathfrak P$  increases by one. These  {\em backward click times} (see Definition~\ref{backwardclicks} and Proposition \ref{thm48}) turn out to be asymptotically  Poisson as $N\to \infty$, with the rate appearing in Theorem \ref{thmclicks}. A key step in proving this is the insight that,  between backward click times, the size of  $\bar {\mathscr A}_{u-r}$, $r \ge 0$ is an autonomous Markov process with the ``logistic'' jump rates \eqref{Z0rates}. As $N\to \infty$, the expected time to extinction out of the quasi-equilibrium of this Markov process is given by \eqref{order_time_ext} up to logarithmic equivalence, see Lemma~\ref{exp_rate}. 
\item Another key result is Proposition \ref{lem_merging1}. This ensures in particular that the ``load zero'' ASG of the entire population $\mathfrak P_T$ that lives at some time $T< u$ (i.e. that part of the ASG which has $\mathcal M$ - distance zero from $\mathfrak P_T$)  merges quickly with the minimum load ASG $\bar {\mathscr A}$.  This helps to show that backward and forward click times are asymptotically close on a suitable scale, and to complete the proof of Theorem~\ref{thmclicks} in Section \ref{secclickrates}.
\end{itemize} 

It remains to  explain our strategy of proving Theorem \ref{thmprofile}. For this, let us have a  look at the random number (say, $K$) of deleterious mutations which an individual sampled uniformly at random at time $t_N$ has acquired in its recent past, i.e. before the ASG of the sampled individual starts to exhibit coalescences and to interact with the ASG's of other individuals sampled uniformly at random at time $t_N$. 

\begin{itemize}
\item During this time span the ASG of the sampled individual looks like a Yule tree with splitting rate $s_N$. We are thus facing a question of  `first passage perco\-lation' along a Yule tree that is decorated with a Poisson point process of intensity $m_N$. This question, which is of independent interest (see Remark \ref{geome}) is the subject of Proposition \ref{pi_eqals_p}. Here, the probability weights $(p_k)_{k \in \N_0}$ appear in a natural manner as the distribution of the above described random variable $K$.
\item  The link to identify them as the quasi-stationary type frequency profile appearing in Theorem \ref{thmprofile} is the insight that the ASG of the sampled individual, as soon as it has become large, contains, with high probability as $N\to \infty$, an individual of the currently best class. This handshake between the backward and the forward point of view is carried out in Section \ref{YuleinASG}, thus completing the proof of the main part of Theorem \ref{thmprofile}.
\end{itemize}

\begin{rem} Another route of proving Theorem \ref{thmprofile}a), which we only sketch briefly here, builds on the time reversibility of the so called equilibrium Ancestral Selection Graph which was discovered by Pokalyuk and Pfaffelhuber in~\cite{PP13}. This carries over to the time reversibility of the equilibrium ASG decorated with the Poisson process of mutations.  As described in Section  \ref{processZ}, the (joint) center of attraction of the sizes of the load $k$-ASG's is asymptotically given by $(2 N s_N p_k)_{k\in \N_0}$. By time reversibility, this backward-in-time concept turns into a statistics of the $\mathcal M$-distances between the equilibrium ASG's at times $0$ and $t_N$, showing that the type frequency profile {\em within the equilibrium ASG at time $t_N$} is asymptoticallly given by $(p_k)_{k\in \N_0}$. This,  however, is representative for the entire population at time $t_N$, since the equilibrium ASG at time $t_N$ is a random sample that is ``measurable from the future''.
\end{rem}

\section{Graphical representation of the model}\label{GRep}
The type frequency process $\xi^{(N)}$ of the $\varphi$-ratchet, which was introduced at the beginning of Section \ref{MMR}, can be constructed (in a similar way as in \cite{casanova2018duality,CHS19,GS}) on top of a Moran graph with selection parameter $s_N$, with mutations added  by means of
an independent Poisson process. 
\begin{defi}[Graphical elements]\label{Graphel} For fixed $N \in \mathbb N$, we consider three independent Poisson point processes, $\mathcal{C}^{(N)},$ $\mathcal{S}^{(N)}$ and $\mathcal{M}^{(N)}$. The processes $\mathcal C^{(N)}$ and $\mathcal S^{(N)}$ 
are supported by $ \{(i,j): i, j \in [N], i\neq j\} \times \mathbb R$ and have on each component $\{(i,j) \}\times \mathbb R$  the constant intensity $\frac1{2N}$ 
and $\frac{s_N}N$, respectively. The third process, $\mathcal 
M^{(N)}$, is a Poisson point process on $G^{(N)}:=[N]\times \mathbb R$  with constant intensity $m_N$ on each component $\{i\}\times \mathbb R$. Here and below we use the abbreviation $[N] := \{1,2,\ldots, N\}$
\end{defi}
\begin{rem}
When there is no risk of confusion, we will suppress the index $N$ and write $G, \mathcal C, \mathcal S, \mathcal M$. We will speak of the points $(i,t) \in G $, $i \in [N]$, as {\em the individuals living at time $t$}. Each point $(i,j,t) \in \mathcal C \cup S$ can be visualized as an arrow pointing from line~$i$ to line $j$ at time~$t$. At an $(i,j,t) \in \mathcal C$, the individual $(i,t)$, irrespective of its type, bears a daughter $(j,t)$ who replaces the individual $(j,t-)$. At an $(i,j,t) \in \mathcal S$, the same happens, but only provided the individual $(i,t)$ carries less mutations than the individual  $(j,t-)$. 
 The process $\mathcal M$ describes the mutations occurring along the lines; each point of $\mathcal M$ increases the mutational load along the lineage by 1. This is made precise in Definition \ref{typetransport}, and illustrated in Figure \ref{fig:boat1}.
\end{rem}
\begin{figure}[h!]
\begin{center}
  \includegraphics[width=.5\linewidth]{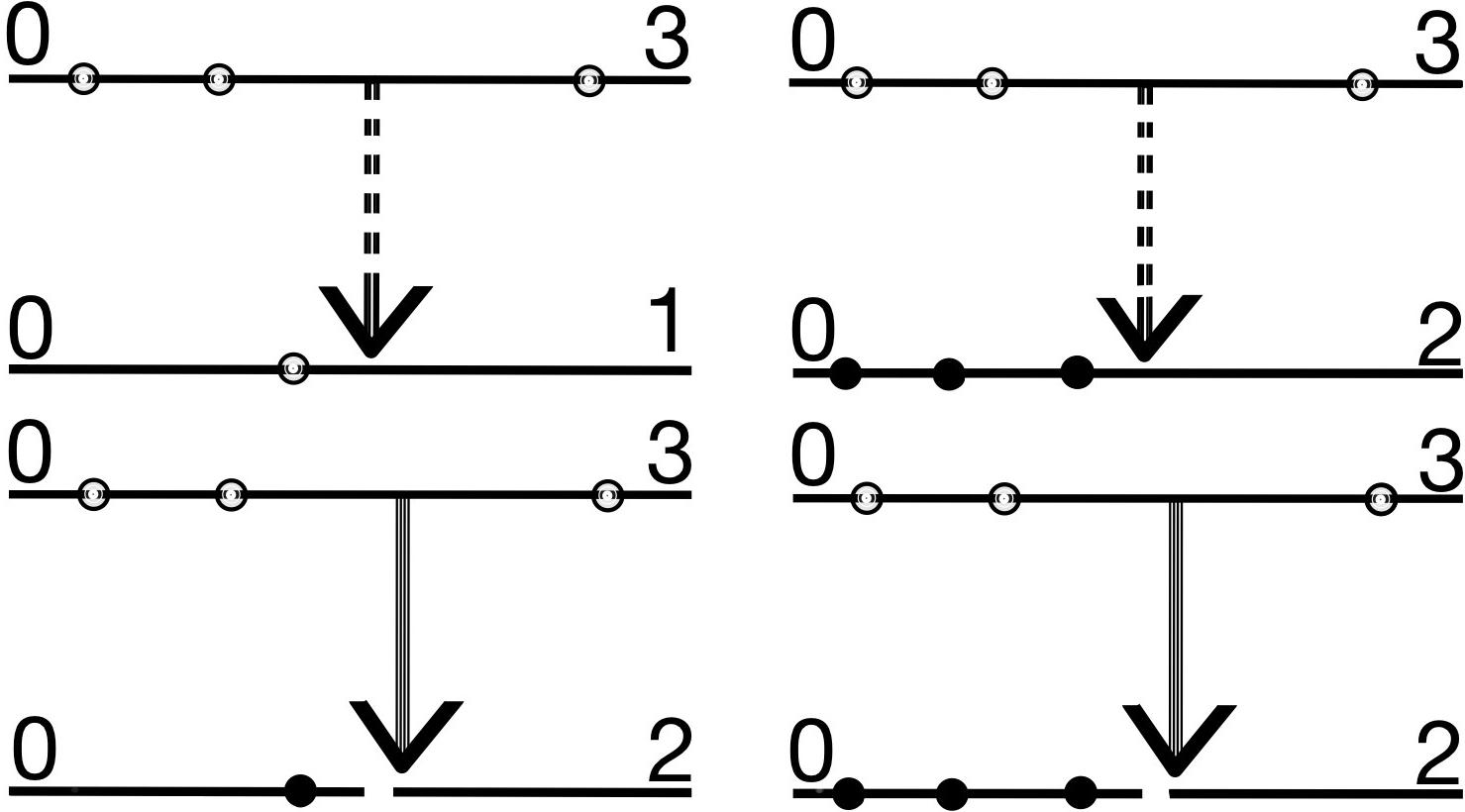}
  \end{center}
  \caption{Graphical elements and their impact on the transport of types. Time is running from left to right, and in each of the four panels two levels ($i=1,2$) are considered, with the initial type configuration $(0,0)$.  Selective and neutral arrows are drawn with dashed and solid shafts, respectively. Mutations are drawn as circles. Because of the rules described in Definition \ref{typetransport}, some of the mutations do not have an effect on the outcome of the types at the final time; these mutations are represented as filled black circles.}
  \label{fig:boat1}
\end{figure}
\begin{defi}[Type configurations and their transport]\label{typetransport}
A {\em type configuration} is an element of ${\mathbb N_0}^{[N]}$, thus assigning a (nonnegative integer) {\em type} to each $i\in [N]$. The process $(\mathcal C, \mathcal S, \mathcal M)$ gives rise to a transport of {\em type configurations} $(\eta(i,t))_{i\in [N]}$ as $t\in \mathbb R$ increases. Specifically, think of an initial type configuration $(\eta(i,t_0))_{i\in [N]}$ being given at some  time $t_0$. At times $t>t_0$, the three Poisson point processes $\mathcal C, \mathcal S$ and $\mathcal M$ act as follows:
\begin{itemize}
\item if the point  $(i,j,t)$ belongs to $\mathcal{C}$, then 
$$\eta(j,t)=\eta(i,t-) \ \  (=\eta(i,t) \ a.s.)$$
\item if the point  $(i,j,t)$ belongs to $\mathcal{S}$, then
$$ \eta(j,t)= \left\{ \begin{array}{lll}
\eta(j,t-) & \text{if} & \eta(j,t-)\leq\eta(i,t-)\\
\eta(i,t-) & \text{if} & \eta(i,t-)<\eta(j,t-)
\end{array}  \right. $$ 
\item if the point $(i,t)$ belongs to $\mathcal{M}$, then
$$\eta(i,t)=\eta(i,t-)+1\, .$$
\end{itemize}
\end{defi}

In the next section and thereafter we will use the Poisson point processes $(\mathcal C, \mathcal S, \mathcal M)$ also for a transport (of potential ancestral paths and mutational loads) {\em backwards in time}. In order to clearly distinguish between forward and backward concepts, we define two filtrations generated by $(\mathcal C, \mathcal S, \mathcal M)$.
\begin{defi}[Forward and backward filtrations]\label{filt}
For $t \in \R$ let $\mathcal C_{\le t}$ and $\mathcal S_{\le t}$  be the restrictions of $\mathcal C$ and $\mathcal S$ to $\bigcup\limits_{i, j \in [N], i\neq j} \{(i,j) \}\times (-\infty, t]$, and let $\mathcal M_{\le t}$ be the restriction of $\mathcal M$ to $\bigcup\limits_{i \in [N]} \{i \}\times (-\infty, t]$. Likewise, define $\mathcal C_{\ge t}$, $\mathcal S_{\ge t}$, $\mathcal M_{\ge t}$, replacing $(-\infty,t]$ by $[t,\infty)$. Let
 $\mathscr F_t$ and $\mathscr P_t$ be the $\sigma$-algebras generated by $\mathcal C_{\le t}, \mathcal S_{\le t}, \mathcal M_{\le t}$ and $\mathcal C_{\ge t}, \mathcal S_{\ge t}, \mathcal M_{\ge t}$, respectively. The {\em forward filtration} is $\mathscr F:= \left( \mathscr F_t\right)_{t\ge 0}$ and the {\em backward filtration} is $\mathscr P:= \left( \mathscr P_t\right)_{t\in \R}$; note that  $\mathscr P_t$ increases as $t$ decreases. A random time $T$ is called a $\mathscr P$-stopping time if $\{T\ge t\} \in \mathscr P_t$ for all $t \in \mathbb R$. The $\sigma$-algebra $\mathscr P_T$ consists of all those sets $E \subset \sigma\left(\bigcup\limits_{t\in \mathbb R} \mathscr P_t\right)$ for which $E \cap \{T \ge t\} \in \mathscr P_t$ for all $t \in \mathbb R$. All these objects are understood for fixed population size $N$; sometimes we will write $\mathscr P^{(N)}$ and  $\mathscr F^{(N)}$ instead of $\mathscr P$ and  $\mathscr F$, to make the dependence on $N$ explicit.
\end{defi}
\begin{rem}  With $t_0:=0$ and $\eta(i,0) :=0$, $i \in [N]$, and with $\eta(.,t)$, $t\ge 0$,  constructed according to Definition \ref{typetransport}, the  process of type frequency evolutions  that figures in Theorems~\ref{thmclicks} and \ref{thmprofile} can now be represented as
the $\mathscr F$-adapted process
\begin{equation}\label{defxi}
 \xi_\kappa(t):= \frac{1}{N}  \# \left\{ i \in [N]: \eta(i,t)=\kappa  \right\}, \, \kappa \in \N_0,  \quad t\ge 0. 
 \end{equation}
Indeed this process has the jump rates given in \eqref{mut}, \eqref{sel}, \eqref{coal}. In terms of $\eta$, the best type at time $t$ (defined in \eqref{besttypexi}) has the representation
\begin{equation}\label{repKN}
K_N^\ast(t) = \min \{\eta(j,t): j\in [N]\}.
\end{equation}
\end{rem}

\section{Potential ancestral paths and their loads}\label{PAP}
While the graphical representation given in the previous section was a {\em forward in time} construction, we now take a {\em backward in time} point of view. This is based on the concept of {\em potential ancestral lineages} which goes back to the pioneering work of  Krone and \mbox{Neuhauser~\cite{KN97,neuhauser1997genealogy}}. The key idea is to construct in a first stage an untyped version of the (potential) genealogy backwards in time and decide in a second stage forwards in time which lineages become ``real''.  Specifically, a ``selective arrow'' $(i,j,t) \in \mathcal S$ introduces the two potential parents $(i,t-)$ and $(j,t-)$ of the individual $(j,t)$. Thus, a potential ancestral lineage backwards in time should jump from $(j,t)$ to $(i,t-)$ as soon as it ecounters the head $j$ of a ``neutral arrow'' $(i,j,t)\in \mathcal C$, and should branch into two selective lineages as soon as it ecounters the head $j$ of a ``selective arrow'' $(i,j,t)\in \mathcal S$. We will formalize this by the concept of {\em (potential ancestral) paths}. 
\begin{defi}[Paths and potential ancestors]\label{pathsdefi}  Let $(i,t_0),\, (j,t) \in G$ with $t_0 < t$. \\
A  {\em (potential ancestral) path} connecting $(i,t_0)$ and $(j,t)$ is a subset of $G$ of the form
$$(\{i_1\} \times [t_0,t_1)) \cup(\{i_2\} \times [t_1,t_2)) \cup... \cup (\{i_n\} \times [t_{n-1},t_n]),$$
with $n \in \N$ and the following properties
\begin{itemize}
\item[a)]
 $t_0 < t_1 < \cdots < t_{n-1} \le t_n=t$,
 \item[b)]
 $i=i_1, \, j= i_n$,
 \item[c)]
  $(i_g, i_{g+1}, t_g) \in \mathcal C \cup \mathcal S$ for $g=1,\ldots, n-1$,
 \item[d)]
 $\mathcal C\left(\left([N]\setminus\{i_g\}\right)\times\{i_g\}\times\left(t_{g-1},t_g\right)\right) = 0$ for $g=1,\ldots, n-1$.
 \end{itemize}
We write $(i,t_0)\prec (j,t)$ if there is a path connecting $(i,t_0)$ and $(j,t)$. In this case we say that $(i,t_0)$ is a {\em potential ancestor} of $(j,t)$.
 \end{defi}
In words, the conditions mean that jumps between different levels $h,h' \in [N]$ may only occur at time points of either neutral or selective arrows, and that none of the time intervals $\left(t_{g-1},t_g\right)$, $g=1,\ldots,n$, may be hit by a neutral arrow whose arrow-head is at $i_g$.

As a consequence of this definition  we observe:
\begin{itemize}
\item If the point $(i,j,t)$ belongs to $\mathcal{C}$, the point $(j,t)$ is disconnected with $(j,t-)$ and connected with $(i,t-)$. 
\item  If the point $(i,j,t)$ belongs to $\mathcal{S}$, the point $(j,t)$ is connected both with $(i,t-)$ and~$(j,t-)$.
 \end{itemize}

 \begin{defi}[Ancestral selection graph (ASG)]\label{ASGdefi}
 For $t \in \mathbb R$ and $J_t \subset [N] \times \{t\}$ we define, suppressing the index $N$, for $r\ge 0$,
$$\mathscr A^{J_t}_{t-r} :=  \{(i,t-r) :  (i,t-r) \prec v \mbox{ for some } v \in J_t\} \quad
\text{and} \quad \mathscr A^{J_t}:= \bigcup_{r\ge 0}\mathscr A^{J_t}_{t-r}.$$
Thinking of  $\mathscr A^{J_t}$ as a union of paths jointly with the graphical elements from $\mathcal C$ and $\mathcal S$ by which it was induced,
we call $\mathscr A^{J_t}$ the {\em ASG back from} $J_t$. \\
For a singleton $J_t = \{(j,t)\}$ we write $\mathscr A^{j,t}_{t-r}$ instead of $\mathscr A^{\{(j,t)\}}_{t-r}$, and for $J_t = [N]\times \{t\}$ we briefly write $\mathscr A^t_{t-r}$ instead of $\mathscr A^{[N]\times \{t\}}_{t-r}$, and $\mathscr A^t$ instead of $\mathscr A^{[N]\times \{t\}}$. 
\end{defi}
\begin{defi}[Load and $\mathcal M$-distance] \label{defi43}
(i) The {\em load} of a path is the number of points of~$\mathcal{M}$ carried by the path.
\\ (ii)
The {\em $\mathcal{M}$-distance} $d_\mathcal M((i,t_0),(j,t))$ of two points $(i,t_0)$, $(j,t) \in G$ with $t_0<t$ is the minimal load of all paths connecting them, with the convention that the minimum over an empty set is infinity.  We say that $(i,t_0)$ is a {\em load $k$ potential ancestor} of $(j,t)$ if $d_{\mathcal M}((i,t_0), (j,t))= k$.\\
(iii) For  $t_0 < t$ and $I_{t_0} \subset [N] \times \{t_0\}$, $J_t \subset [N] \times \{t\}$ we put
$$d_{\mathcal M}(I_{t_0}, J_t) := \min\{d_{\mathcal M}(v, w): v \in I_{t_0}, w \in J_t\}.$$
\end{defi}
\begin{rem} For any three points $(i,t_0), (j,t), (g,u) \in G$ with $t_0< t < u$ one has
$$d_{\mathcal M}((i,t_0), (g,u)) \le d_{\mathcal M}((i,t_0), (j,t))  + d_{\mathcal M}((j,t), (g,u)).$$
To see this we may assume  w.l.o.g. that  $(i,t_0) \prec (j,t)$ and $(j,t) \prec (g,u)$ (since otherwise the r.h.s. of the inequality is infinite). Then the concatenation of a path of minimal load connecting $(i,t_0)$ and $(j,t)$ with a path   of minimal load connecting $(j,t)$ and $(g,u)$ is a path connecting $(i,t_0)$ and $(g,u)$; hence $d_{\mathcal M}$ obeys the claimed triangle inequality.
\end{rem}
\begin{figure}[h!]
\begin{center}
  \includegraphics[width=.5\linewidth]{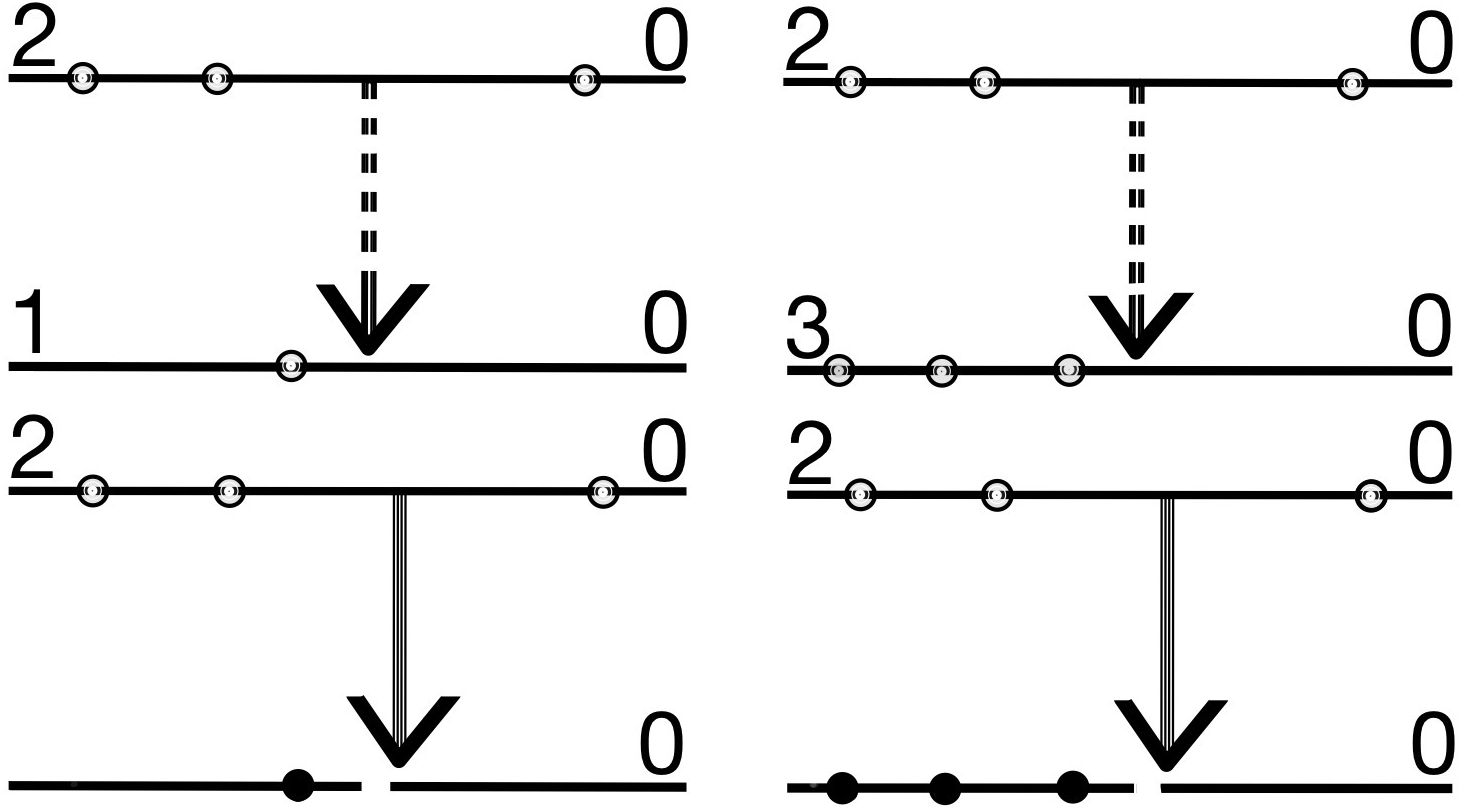}
  \end{center}
  \caption{This figure contains the same graphical elements than Figure \ref{fig:boat1}, but now the paths are followed backwards. Let us think of the left hand side of each of the four panels corresponding to time $t_0$ and the right hand side corresponding to $t>t_0$. At time $t_0$, the $\mathcal M$-distance between the  set $\{1,2\}\times \{t\}$ and its potential ancestors at time $s$ is annotated. A comparison with Figure~\ref{fig:boat1} shows differences and similarities between the backward and forward transport. }
  \label{fig:boat2}
\end{figure}
The graphical construction entails that, for $t_0 < t$, the type configuration $\eta(.,t)$ is a function of  $\eta(.,t_0)$ and the graphical elements between times $t_0$ and $t$: types are transported forward in time, and whenever there is a ``selective encounter'' between two ancestral paths of $(j,t)$, the  ``better type'' is passed on. Specifically, this results in 
\begin{rem} [Flow of type configurations] Let $\eta = \eta^{(N)}$ be as specified in Definition \ref{typetransport}. Then for any $j \in [N]$, $0\le t_0 <t$, one has a.s.
$$\eta(j,t)= \min_{i \in [N]} \left\{ \eta(i,t_0)+d_{\mathcal M}((i,t_0),(j,t)) \right\}.  $$
In particular, if $\eta(i,0) =0$ for all $i \in [N]$, then 

\begin{equation}\label{etadist}
\eta(j,t) = d_{\mathcal M}([N] \times \{0\}, (j,t)) =  d_{\mathcal M}(\mathscr A^{j,t}_0, (j,t))
\end{equation}
and
\begin{equation} \label{besttype}
K_N^\ast(t) = \min \{\eta(j,t): j\in [N]\} = d_{\mathcal M}\left([N] \times \{0\}, [N] \times \{t\}\right).
\end{equation}
\end{rem}
\begin{defi} \label{colouredASG} 
a) {\rm (Load $k$ potential ancestors)} Let $k \in \mathbb N_0$, $r\geq 0$ and $t\in \mathbb R$. For $J_t\subset [N]\times \{t\}$, 
$\mathscr{A}^{J_t}_{t-r}(k)$ is the set of individuals $(i,t-r)$  which are load $k$ potential ancestors  of some individual in $J_t$ (cf. Definition \ref{defi43}).
Taking the union over $r \ge 0$ we define
$\mathscr A^{J_t}(k)$ as the set of all individuals which are load $k$  potential ancestors of some individual \mbox{in $J_t$.}

b) {\rm Minimum load potential ancestors}. We define the  set of {\em minimum load potential ancestors} at time $t-r$ of the population $J_t$  as 
\begin{equation} \label{min_load_anc} \bar{\mathscr A}^{J_t}_{t-r}:= \mathscr{A}^{J_t}_{t-r}(\underline k)
\end{equation}
where 
$$ \underline k := \underline k(t-r,t) := \min\{k\in \N_0:  \mathscr{A}^{J_t}_{t-r}(k)\neq \emptyset\}.$$
To ease notation we write $\bar{\mathscr A}^t_{t-r}$ instead of  $\bar{\mathscr A}^{[N]\times\{t\}}_{t-r}$, and $\bar{\mathscr A}^{i,t}_{t-r}$ instead of $\bar{\mathscr A}^{\{(i,t)\}}_{t-r}$.

c) The definitions in a) and b) extend directly from deterministic $t$ and $J_t$ to a $\mathscr P$-stopping time $T$ and a $\mathscr P_T$-measurable random set $\mathscr J_T \subset [N]\times \{T\}$.
\end{defi}
\section{Percolation of loads along the Ancestral Selection Graph} \label{PercASG}
In this section we fix a population size $N \in \mathbb N$ which we suppress in the notation. For a $\mathscr P$-stopping time $T$ and a $\mathscr P_T$-measurable set $\mathscr J_T \subset [N]\times \{T\}$  (cf. Definition \ref{filt}), the joint dynamics of the set-valued processes $\big(\mathscr{A}^{\mathscr J_T}_{T-r}(k)\big)_{r\ge 0}$, $k \in \N_0$, as specified in Definition~\ref{colouredASG}, is driven in a $\mathscr P$-adapted manner by the Poisson point processes $(\mathcal{C},\mathcal{S},\mathcal{M})$ introduced in Definition \ref{Graphel}. With regard to Definitons \ref{pathsdefi} and \ref{ASGdefi} we will now describe  the actions of $(\mathcal{C},\mathcal{S},\mathcal{M})$ on the sets $\mathscr{A}^{J_t}_{t-r}(k)$  (note the analogy and the differences to Definition \ref{typetransport} for the transport of type configurations,  which there was forward in time). For $r\ge0$ and $ k\in \N_0$ we define $H$, $H_k \subset [N]$ by
$$\mathscr{A}^{\mathscr J_T}_{T-r}= H\times \{T-r\} \quad \text{and} \quad \mathscr{A}^{\mathscr J_T}_{T-r}(k) = H_k\times \{T-r\}.
$$
\begin{itemize}
\item[$\bullet$] {\em Coalescences:} Let $(i,j,T-r)\in \mathcal{C}$. 
If $i \in H_k$ and $j \in H_{k'}$, then for $k\le k'$
$$H_k \mbox{ remains unchanged and  } j \mbox{ is removed from } H_{k'}$$
whereas for $k > k'$
$$i \mbox{ changes from } H_k \mbox{ to } H_{k'}, \, \mbox{ and } j \mbox{ is removed from } H_{k'}.$$
\item[$\bullet$] {\em Selective branching:} Let $(i,j,T-r)\in \mathcal{S}$.
\\ If $j \in H_k$ and $i \notin H$, then \,$i \mbox{ becomes an element of } H_k.$
 \item[$\bullet$] {\em Selective competition:} Let $(i,j,T-r)\in \mathcal{S}$. 
 If $i \in H_k$ and $j \in H_{k'}$, then for $k\le k'$
 $$j \mbox{ changes from } H_{k'} \mbox{ to } H_k$$
whereas for $k > k'$
$$i \mbox{ changes from } H_k \mbox{ to } H_{k'}.$$
 \item[$\bullet$] {\em Mutation:} Let $(i,T-r)\in \mathcal{M}$. 
If  $i \in H_k$ then
$i \mbox{ changes from } H_k \mbox{ to } H_{k+1}.$
\end{itemize}
Due to the symmetry properties of the dynamics (backwards in time) that is induced by the just described transitions, we may focus our attention on the configuration of cardinalities of the sets $\mathscr{A}^{\mathscr J_t}_{T-r}(k)$, and define
$$ {A}^{\mathscr J_T}_{T-r}(k) := \# \mathscr{A}^{\mathscr J_T}_{T-r}(k), \quad k \in \N_0. $$
For $\mathscr J_T = [N]\times \{T\}$ we write $ A^T_{T-r}(k)$ instead of $ A^{[N]\times \{T\}}_{T-r}(k)$, and $ A^T(k)$ instead of $ A^{[N]\times \{T\}}(k)$.

The following lemma is a consequence of the above described actions of the Poisson point processes $(\mathcal C, \mathcal S, \mathcal M)$  on the sets $\mathscr{A}^{\mathscr J_T}_{T-r}(k)$. 

\begin{lem}\label{fromAtoZ} For $T$ and $\mathscr J_T$ as in Definition~\ref{colouredASG} c), the process 
$$
(A^{\mathscr J_T}_{T-r}(0), A^{\mathscr J_T}_{T-r}(1), \ldots, A^{\mathscr J_T}_{T-r}(k),...)_{r\ge 0}$$
 is Markovian when randomized over $(\mathcal{C},\mathcal{S},\mathcal{M})$. Its state space is the set
 \begin{equation}\label{defZN}
 \mathscr Z_N := \{z=(z_0,z_1, \ldots, ): z_k \in \mathbb N_0,  z_0+z_1+ \cdots \le N\}.
 \end{equation}
Its jump rates from $z\in \mathscr Z_N$ are (with $e_k$ as in Section \ref{MMR} and $s_N$, $m_N$ as in \eqref{moderatesm})
\begin{itemize}
\item[$\bullet$] Coalescences: for any $k \in \mathbb N_0$, 
\begin{equation}\label{ratescoal}z \to z -e_{k} \quad \text{with rate} \quad \frac 1{2N} z_k (z_k -1)+ \frac 1N  z_k \sum_{0\le k'<k}z_{k'}.
\end{equation}
\item[$\bullet$] Selective branching: for any $k \in \mathbb N_0$,
\begin{equation} \label{ratesselbranch} z \to z + e_k \quad \text{with rate} \quad \frac{s_N}N z_k \left(N-\sum_{0\le k'<\infty} z_{k'}\right). 
\end{equation}
\item[$\bullet$] Selective competition:  for any pair of integers $(k,k')$ such that $0\le k < k'$,
\begin{equation} \label{ratesselcomp} z\to z + e_k-e_{k'} \quad \text{with rate} \quad \frac{s_N}N z_kz_{k'}. 
\end{equation}
\item[$\bullet$] Mutation: for any $k\in \mathbb N_0$,
\begin{equation*}z \to z + e_{k+1}-e_{k} \quad \text{with rate} \quad m_N  z_k 
\end{equation*}
 \hspace{1.6cm} and for any $k \in \mathbb N$,
\begin{equation}\label{ratesmutfeeding}z \to z + e_{k}-e_{k-1} \quad \text{with rate} \quad m_N  z_{k-1}. 
\end{equation}
\end{itemize} 
\end{lem}
\begin{figure}[h!]
 \center \includegraphics[width=.6\linewidth]{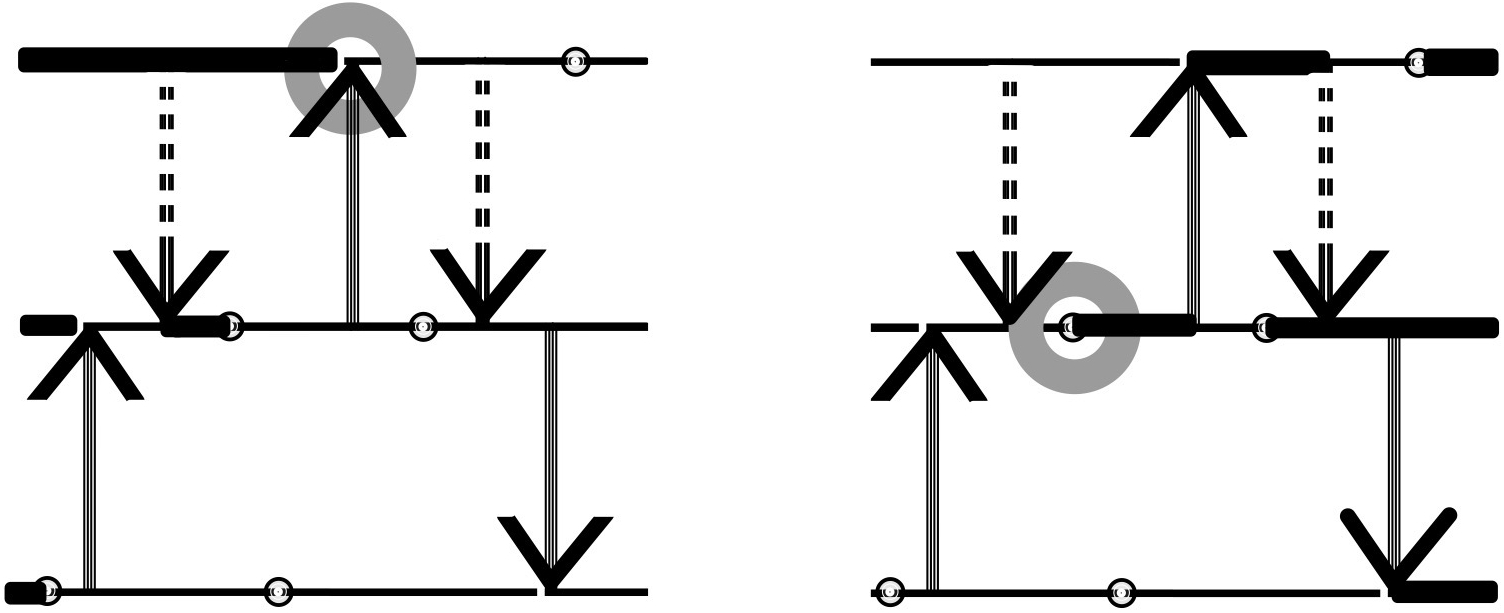}
  \caption{Key to the analysis of the click times of the ratchet are the instances at which (seen backward in time) the paths with minimal load are lost. In this figure we observe how the click times forward and backward are different, but strongly related to each other and also close in time. Both panels contain the same graphical elements, with the left panel showing the forward transport and the right panel showing the backward transport of $\mathcal M$-distances. In each case, the points with $\mathcal M$-distance 0 from the left respectively from the right boundary are shown by thick lines, and \textit{clicks} are indicated by a large circle.}
  \label{fig:boat3}
\end{figure}
\section{A hierarchy of logistic competitions} \label{processZ}
Throughout this section we consider, for any given $N \in \mathbb N$, a Markovian jump  process $Z^{(N)}:= (Z^{(N)}(r))_{r\ge 0}$, whose state space is  $\mathscr Z_N$ defined in \eqref{defZN} and whose jump rates are given by \eqref{ratescoal} to \eqref{ratesmutfeeding}. When there is no risk of confusion, we suppress the superscript $N$ and write e.g. $Z_0$ instead of $Z^{(N)}$. The following remark stems from the jump rates \eqref{ratescoal} to \eqref{ratesmutfeeding}.
\begin{rem} \label{cor_rates_01}
For each $k\in \N_0$ the process $(Z_0, \ldots, Z_{k})$ is Markovian with jump rates given by \eqref{ratescoal} to~\eqref{ratesmutfeeding}. In particular:\\
a) The process $Z_0$ is  Markovian and jumps with the rates
\begin{eqnarray}\begin{split}\label{Z0rates}
&z_0 \to z_0+1 \quad \text{ at rate} \quad s_Nz_0\left(1-\frac{z_0}{N}\right)=:z_0\,b_0(z_0,N), \\
&z_0 \to z_0-1 \quad \text{at rate} \quad z_0 \left(m_N+\frac{z_0-1}{2N} \right)=:z_0\, d_0(z_0,N) .
\end{split}
\end{eqnarray}
b) The process~$(Z_0,Z_1)$ is Markovian and jumps with the rates
\begin{eqnarray}\begin{split}
&(z_0, z_1)\to (z_0,z_1+1) \quad \text{at rate} \quad s_Nz_1\left(1-\frac{z_0+z_1}{N}\right), \\
&(z_0, z_1)\to (z_0-1,z_1+1) \quad \text{at rate} \quad m_Nz_0, \\
&(z_0,z_1)\to (z_0,z_1-1) \quad \text{at rate} \quad \frac{z_1(z_1-1)}{2N}+z_1 \left(m_N+\frac{z_0}{N} \right) \\
&(z_0, z_1)\to (z_0+1,z_1-1) \quad \text{at rate} \quad s_N\frac{z_0z_1}N.
\end{split}
\end{eqnarray}
\end{rem}
\begin{rem}\label{n_k}
a) An inspection of the rates  \eqref{ratescoal} to~\eqref{ratesmutfeeding} and an application of a dynamical law of large numbers  \cite[Theorem 11.3.2]{EK86} show that if $Z^{(N)}_0(0)$
is of order $N/f(N)$, the process 
$(({f(N)}/{N})Z^{(N)}\big(f(N)r\big))_{r \geq 0}$
 is on each time interval $[0,r_0]$  for large $N$   close (uniformly in $r \in [0,r_0]$) to the solution of the dynamical system
\begin{equation} \label{syst_dyn} \frac{dn_k(t)}{dt}= \mu n_{k-1}(t)+ n_k(t)\left( \alpha-\mu -\frac{n_k(t)}{2}- \sum_{i=0}^{k-1}n_i(t) \right), \quad k \geq 0 \end{equation}
with $n_{-1} := 0$.
Without going into all details here, let us mention that two steps are needed to prove this convergence. First, we consider a modified version of the process $Z^{(N)}$, namely $\widetilde{Z}^{(N)}$, where the rates in \eqref{ratesselbranch} and \eqref{ratesselcomp} are replaced by $s_N z_k$ and $0$ respectively. Choosing as the mass rescaling parameter the carrying capacity $N/f(N)$, we can directly apply Theorem 11.2.1 in \cite{EK86} to the process $(\widetilde{Z}^{(N)}(f(N)r))_{r \geq 0}$. Then applying Lemma C.1 in \cite{chazottes2016sharp} as in the proof of Lemma \ref{lem_coupling1}, we obtain that the sum of the components of the process $(\widetilde{Z}^{(N)}(f(N)r))_{r \geq 0}$ does not reach a size of order $N$ within a time of order $\ln N$ with a probability close to $1$ for large~$N$. The modification of the jump rates is thus negligible on a time scale of order $1$, and the claimed convergence  holds for the process $(Z^{(N)}(f(N)r))_{r \geq 0}$.\\
b) The system \eqref{syst_dyn} has a unique attracting equilibrium $(\bar{n}_k)_{k \in \N_0}$ which follows the recursion
\begin{equation}\label{recursion_n}
\bar{n}_0 := 2(\alpha-\mu)\quad {\rm and}  \quad
 \mu \bar{n}_{k-1} + \bar{n}_k \left(\alpha-\mu  - \frac{\bar{n}_k}{2}-\sum_{i=0}^{k-1}\bar{n}_i \right)=0,  \quad k \geq 1. 
\end{equation} 
For each $k \in \N_0$ the process $(Z^{(N)}_0,...,Z^{(N)}_k)$ has a {\em quasi-stationary equilibrium} (see e.g. \cite[Example 5.4]{champagnat2023general})  which we denote by  $\nu^{(k)}_N$.  
We thus obtain that for any $r\geq 0$,
$$ \lim_{N \to \infty}\E_{\nu^{(k)}_N} \left[ \frac{f(N) Z^{(N)}_k(r)}{N} \Big| Z^{(N)}_0(r)\geq 1 \right]= \bar{n}_k. $$

Summing over $k$ in \eqref{recursion_n} and defining $\displaystyle \bar n:=\sum_{k=0}^\infty \bar n_k $, we get 
\begin{align*}0&=  \alpha\bar n- \frac{1}{2} \left(\sum_{k=0}^\infty \bar n^2_k+ 2\sum_{0\le i < k < \infty} \bar n_k\bar n_i \right)= \alpha \bar n- \frac{1}{2}  \bar n^2. \end{align*}
This yields
\begin{equation}\label{nbar}
\bar{n}= \sum_{k=0}^\infty \bar{n}_k= 2\alpha. 
\end{equation}
c)
Let $(\bar n_k)_{k\in \N_0}$ be defined by the recursion~\eqref{recursion_n}. In view of \eqref{recursion_n} and \eqref{nbar} it is clear that 
$\big(\frac {\bar n_k}{2\alpha}\big)_{k \in \N_0}$ is 
 a sequence of probability weights which satisfies the recursion \eqref{recursion} and thus coincides with the probability weights $(p_k)_{k \in \N_0}$, that are defined in Theorem~\ref{thmprofile}a).
 \end{rem}
 
The next lemma roughly says that for any $k\in \mathbb N_0$ the process $Z^{(N)}_k$  with high probability grows  quickly to a size of order $N/f(N)$ and stays there at least for a time of order $f(N) \ln N$, provided only that for some $\ell \le k$ the initial size of $Z^{(N)}_\ell$ is not too small. In view of Lemma~\ref {fromAtoZ}, the quantity $N/f(N)$ thus characterizes the typical size of the ASG on the $f(N) \ln N$ - timescale.
\begin{lem}\label{lem_coupling1}
Let $(\bar{n}_k)_{k \in \N_0}$ be given by the recursion \eqref{recursion_n}.
 Let $(Z_0,Z_1,...)=(Z_0^{(N)},Z_1^{(N)},...)$ be a process with jump rates given by \eqref{ratescoal} to \eqref{ratesmutfeeding}, and let $R>0$. Then for any $k \in \N_0$ and $\eps>0$, there exist  finite constants $C_k$ and $C_k(\eps)$ such that 
\begin{eqnarray} \begin{split}\label{containment_B12}
\liminf_{N \to \infty} &\, \P \Big( \frac{f(N)}{N} Z_k(r) \in \left[ \bar{n}_k-C_k\eps,\bar{n}_k+C_k\eps\right] \mbox{ for all } r \mbox{ such that }\\ & \frac{r}{f(N) \ln N} \in [C_k(\eps),C_k(\eps)+R] \,\Big|\, \exists\,  \ell\leq k \mbox{ such that } Z_{\ell}(0)\geq 1/\eps \Big)=1-\delta(\eps),
\end{split}
\end{eqnarray}
where $\delta(\eps) \to 0$ as $\eps\to 0$.
\end{lem}

The proof of this lemma  will be given in  Section \ref{ProofofLemmas}.
 
Let us now introduce the first time at which $Z_0^{(N)}$ (recall \eqref{Z0rates}) hits the trap $0$:
\begin{equation}\label{hitting_Z0i}
H_0^{(N)}:= \inf \{ r \geq 0,\, Z_0^{(N)}(r)= 0 \}.
\end{equation}

\begin{lem}\label{exp_rate}
Let $\nu_N:= \nu_N^{(0)}$ denote the quasi-stationary distribution of $Z_0^{(N)}$. For every $N$ there exists $\theta_N$  satisfying \eqref{order_time_ext} such that 
\begin{equation} \label{QSD}
\P_{\nu_N} \left( H_0^{(N)} >r  \right) = e^{-r/f(N)\theta_N}, \quad r >0.
\end{equation}
\end{lem}

\begin{lem}\label{lem_fast_ext}
a) There exists $C>0$ such that for $\eps>0$
$$ \liminf_{N \to \infty} \P \left( Z_1^{(N)}(H_0^{(N)})\geq \frac{C}{\eps^2}\, \Big|\,Z_0^{(N)}(0) \geq 1/\eps \right)= 1-\delta(\eps), $$
where $\delta(\eps) \to 0$ as $\eps\to 0$.

b) With $(\theta_N)$ as in Lemma \ref{exp_rate}, let $(r_N)$ be a sequence with $\frac{r_N}{f(N)\theta_N} \to \infty$, and define ${\underline k}_N:= \min\{k \in \mathbb N: Z_{k}^{(N)}(r_N) > 0\}$. Then
$$ \liminf_{N \to \infty} \P \left( Z_{\underline k_N}^{(N)}(r_N)\geq \frac1{\eps}\,\Big |\,Z_0^{(N)}(0) \geq 1/\eps\right) =  1-\delta(\eps), $$
where $\delta(\eps) \to 0$ as $\eps\to 0$.
\end{lem}

The two just stated lemmas (which will be proved in Section \ref{ProofofLemmas})  are key for  obtaining the renewal structure of the dynamics of the potential ancestors with minimal load. They imply in particular that when the set of potential ancestors with the currently minimal load gets extinct, the number of minimum load potential ancestors that ``come next''   
 is large enough for reaching
a size of order $N/f(N)$ given by the quasi-stationary distribution $\nu_N$. As we will see in Section \ref{secclickrates}, this will ensure, using duality, that the succession of several clicks (in the sense of Definition \ref{clicktimes}) within a time frame of order smaller than $f(N)\theta_N$ is not likely.

\section{Quick merging along the Ancestral Selection Graph}\label{Qickmerge}
The main result of this section, which will be a key ingredient in the proofs of  Theorems~\ref{thmclicks} and \ref{thmprofile}a), is an upper estimate for the time it takes for the merging of the sets of  load $k$ potential ancestors of two  $\mathscr P_T$-measurable random sets $ \mathscr J_T^1$ and $\mathscr J_T^2$ of  $[N] \times \{T\}$, where $T$ is a $\mathscr P$-stopping time (recall Definition \ref{filt}). Roughly stated, this result says that this merging happens with high probability as $N\to \infty$ within a time frame of order $f(N)\ln N$, provided only that the sets  $\mathscr J_T^1$ and $\mathscr J_T^2$ are sufficienly large.
With reference to Definition~\ref{colouredASG}, we define 
  the {\em (random) merging time} of the two load $k$ ASG's  $\mathscr{A}^{\mathscr J_T^1}(k)$ and $\mathscr{A}^{\mathscr J_T^2}(k)$ as
\begin{equation} \label{def_merg_ancgraph0}
\mathscr{C}_k^{\mathscr J_T^1, \mathscr J_T^2}:= \sup \left\{ t\leq T: \mathscr{A}^{\mathscr J_T^1}_t(k)=\mathscr{A}^{\mathscr J_T^2}_t(k)\right\}.
\end{equation}
\begin{prop}\label{lem_merging1}Let $T$ be a $\mathscr P$-stopping time and let $\mathscr J_T^1$, $\mathscr J_T^2$ be $\mathscr P_T$-measurable random subsets of $[N]\times\{T\}$. Then,
for any $k \ge 0$ and  $\eps>0$, there exists a finite constant $C(\eps)$ s.t.
\begin{equation}\label{mergingstatement} \liminf_{N \to \infty} \P \left(\mathscr{C}_k^{\mathscr J_T^1, \mathscr J_T^2}\geq T-C(\eps) f(N) \ln N \Big|\#\mathscr J_T^1 \ge 1/\eps \, \mbox{ and } \#\mathscr J_T^2\geq 1/\eps \right)\ge 1-\delta(\eps) 
\end{equation}
with $\delta(\eps) \to 0$ as $\eps \to 0$.
\end{prop}
\begin{proof} 
The strategy of the proof consists in showing by induction that for all $k \ge 0$ the sets 
$$ \mathscr{A}^{\mathscr J_T^1}(0) \cup ... \cup \mathscr{A}^{\mathscr J_T^1}(k) \quad \text{and} \quad \mathscr{A}^{\mathscr J_T^2}(0) \cup ... \cup \mathscr{A}^{\mathscr J_T^2}(k) $$
merge with high probability within a time of order $f(N) \ln N$.
Let us begin with the case $k=0$.
We will write $\mathscr A^i_{T-r}:= \mathscr{A}^{\mathscr J_T^i}_{T-r}(0)$, $r\ge 0$, $i=1,2$, and
 will study the dynamics of the set-valued process $$ \mathscr A^1_{T-r} \triangle \mathscr A^2_{T-r} = \left(\mathscr A^1_{T-r} \cup \mathscr A^2_{T-r}\right) \setminus \left(\mathscr A^1_{T-r} \cap \mathscr A^2_{T-r}\right), \quad r\ge 0,$$ and of its cardinality 
$$ \#\left( \mathscr A^1_{T-r} \triangle \mathscr A^2_{T-r}\right)= \#(\mathscr A^1_{T-r} \cup \mathscr A^2_{T-r}) - \# (\mathscr A^1_{T-r}\cap \mathscr A^2_{T-r})$$
as $r$ increases. For the sake of readability, we define  subsets $H_1$, $H_2$ by
$$\mathscr A^1_{T-r} = H_1\times\{T-r\}, \quad \mathscr A^2_{T-r} = H_2\times\{T-r\}.$$
Five possible types of transitions of $H_1\triangle H_2$ may result from the elements of the processes $(\mathcal C, \mathcal S, \mathcal M)$:
\begin{itemize}
\item $(i,j,T-r) \in \mathcal{S}$ with $j \in H_1 \triangle H_2$ and $i \notin H_1 \cup H_2$; then $i$ becomes an element of $H_1 \triangle H_2.$ 
This type of transition adds one element to $H_1 \triangle H_2$ and has a rate
\begin{equation}\label{r1} q_1:= s_N \#\left(H_1 \triangle H_2\right)\left(1-\frac{\# (H_1 \cup H_2)}{N}\right) 
\end{equation}
\item $(i,j,T-r) \in \mathcal{C}$ with $j \in H_1 \triangle H_2$ and $i \in H_1 \cap H_2$; then $j$ is removed from $H_1 \triangle H_2$.\\
This type of transition removes one element from $H_1 \triangle H_2$ and has a rate
\begin{equation*}
 q_2:=\frac{1}{N} \#\left(H_1 \triangle H_2\right)\#\left(H_1\cap H_2\right) 
 \end{equation*}
\item $(i,j,T-r) \in \mathcal{C}$ with $i$ and $j$ either both belonging to $H_1\setminus H_2$ or both belonging to $H_2\setminus H_1$; then $j$ is removed from $H_1 \triangle H_2$.\\
This type of transition removes one element from $H_1 \triangle H_2$ and has a rate
\begin{equation}\label{r3}
q_3:=\frac{1}{2N}\left( \#\left( H_1 \setminus H_2\right)\left( \#\left( H_1 \setminus H_2 \right)-1\right) +  \#\left( H_2 \setminus H_1\right)\left( \#\left( H_2 \setminus H_1 \right)-1\right)\right).
\end{equation}
\item $(i,j,T-r) \in \mathcal{C}$ with one of $i$ and $j$  belonging to $H_1\setminus H_2$ and the other belonging to $H_2\setminus H_1$; then both $i$ and $j$ are removed from $H_1 \triangle H_2$.\\
This type of transition removes two elements from $H_1 \triangle H_2$ and has a rate 
\begin{equation}\label{r3}
q_4:=\frac{1}{N}\big( 
\#
\left( H_1 \setminus H_2\right)
 \#\left( H_2 \setminus H_1 \right)\big).
\end{equation}
\item $(i,T-r) \in \mathcal{M}$ with $i \in H_1 \triangle H_2$; then 
then $j$ is removed from $H_1 \triangle H_2$.\\
This type of transition removes one element from $H_1 \triangle H_2$ and has a rate
\begin{equation}\label{r4}  q_5=m_N\#\left(H_1 \triangle H_2\right).
\end{equation}
\end{itemize}
The sum of $q_2$, $q_3$, $q_4$ and $q_5$ equals
\begin{align} \notag q_2+q_3+q_4+q_5=&  \#(H_1 \triangle H_2)\left( m_N+ \frac{\#(H_1 \cup H_2) +\# (H_1\cap H_2)-1}{2N}\right)\\ \label{q2to5}
 =& \#(H_1 \triangle H_2)\left( m_N+ \frac{\#H_1+ \#H_2-1}{2N}\right). 
 \end{align}
From Lemma \ref{fromAtoZ} and  Lemma \ref{lem_coupling1} we know that for any $\eps>0$ there exists a constant $C(\eps)$ such that if 
$$ \# \mathscr A^i_T \geq 1/\eps , \,i=1,2,$$
then for any $R>0$, with a probability close to $1$ for $\eps$ small enough and $N$ large enough, 
$$\frac{2N}{f(N)}(\alpha-\mu-2\eps)\leq \# \mathscr A^1_{T-r},\, \#\mathscr A^2_{T-r} \leq  \frac{2N}{f(N)}(\alpha+\mu-2\eps) $$
for  $C(\eps)f(N)\ln N  \leq r \leq (C(\eps)+R)f(N)\ln N$.
Because of~\eqref{r1} and \eqref{q2to5}, in such a time window we have
$$ q_1\leq \frac{\alpha}{f(N)}\left(\#(\mathscr A^1_{T-r} \triangle \mathscr A^2_{T-r})\right) \quad \text{and} \quad
q_2+q_3+q_4+
q_5 
\geq \#(\mathscr A^1_{T-r} \triangle \mathscr A^2_{T-r}) \frac{ \alpha+ (\alpha-\mu-5\eps)}{f(N)}$$
for $N$ large enough.
The process $ \left(\#(\mathscr A^1_{T-r} \triangle \mathscr A^2_{T-r})\right)_{r \ge 0}$ is thus stochastically dominated by a branching process with individual birth rate $\alpha/f(N)$ and death rate $$\left( \alpha+ (\alpha-\mu-5\eps)\right) /f(N). $$
The extinction time of such a process, with an initial state smaller than $N$, is smaller than
$$ 2\frac{f(N)\ln N}{\alpha-\mu-5\eps} $$
with a probability converging to $1$ when $N$ goes to infinity (see e.g. \cite{champagnat2021stochastic} Lemma A.1). This concludes the proof of the proposition for the case $k=0$. 

Assume now that the sets 
$$ \mathscr{A}^{\mathscr J_T^1}(0) \cup ... \cup \mathscr{A}^{\mathscr J_T^1}(k-1) \quad \text{and} \quad
 \mathscr{A}^{\mathscr J_T^2}(0) \cup ... \cup \mathscr{A}^{\mathscr J_T^2}(k-1) $$
merge at time $T_{k-1} =T-R_{k-1}$, where  $R_{k-1} = O(f(N) \ln N)$. From Lemma \ref{lem_coupling1} we know that there exists $R<\infty$ such that for any $K<\infty$, with a probability close to one the size of this union is close to 
$$ \frac{N}{f(N)}\sum_{l=0}^{k-1} \bar{n}_l $$
during the time interval $[T_{k-1}+Rf(N) \ln N,T_{k-1}+(R+K)f(N) \ln N]$, and remains to be so during any time frame of order $f(N)\ln N$. We also know that the sizes of $ \mathscr{A}^{\mathscr J_T^1}(k)$ and $ \mathscr{A}^{\mathscr J_T^2}(k)$ are close to $N\bar{n}_k/f(N)$ during the same time frame. Let us again use the abbreviations  $\mathscr A^1_{T-r}$ and $\mathscr A^2_{T-r}$, now for
$$\mathscr A^i_{T-r}:= \mathscr{A}^{\mathscr J_T^i}_{T-r}(0)\cup... \cup \mathscr{A}^{\mathscr J_T^i}_{T-r}(k), \quad  i=1,2.$$
By definition of $T_{k-1}$ we have the equality
\begin{equation}\label{Acoupling}
 \#\left(\mathscr A^1_{T_{k-1}-r} \triangle \mathscr A^2_{T_{k-1}-r}\right)= \# \left(\mathscr{A}^{\mathscr J_T^1}_{T_{k-1}-r}(k) \triangle \mathscr{A}^{\mathscr J_T^2}_{T_{k-1}-r}(k)\right), \quad r \ge 0. 
 \end{equation}

Another crucial observation is that the upward and downward jump rates of the process
$$\# \left(\mathscr A^1_{T_{k-1}-r}\triangle  \mathscr A^2_{T_{k-1}-r}\right)_{r\ge 0}$$
are the same as those of the process 
$$\# \left(\mathscr{A}^{\mathscr J_T^1}_{T-r}(0)\triangle \mathscr{A}^{\mathscr J_T^2}_{T-r}(0)\right)_{r\ge 0},$$
resulting from \eqref{r1} -- \eqref{r4}.  (In particular, for $T-r\le T_{k-1}$, 
the mutational events only affect the set
$\mathscr{A}^{\mathscr J_T^1}_{T-r}(k) \triangle \mathscr{A}^{\mathscr J_T^2}_{T-r}(k)$, 
whose cardinality by \eqref{Acoupling} equals that of $\mathscr A^1_{T-r} \triangle \mathscr A^2_{T-r}$.)
The rest of the proof now follows that same lines as in the case $k=0$.
\end{proof}

Similarly as in \eqref{def_merg_ancgraph0}, we define 
  the {\em (random) merging time} of the ASG's  $\mathscr{A}^{\mathscr J_T^1} $ and $\mathscr{A}^{\mathscr J_T^2} $ as
$$\mathscr{C}^{\mathscr J_T^1, \mathscr J_T^2}:= \max \left\{ t\leq T: \mathscr{A}^{\mathscr J_T^1}_t=\mathscr{A}^{\mathscr J_T^2}_t\right\}.$$
Since in the special case $m_N = 0$ the load zero ASG $\mathscr A^{\mathscr J_T}_t(0)$ equals the `untyped' ASG $\mathscr A^{\mathscr J}_t$, we immediately obtain the following corollary by putting $\mu=0$ and $k=0$ in Proposition~\ref{lem_merging1}:
\begin{cor}\label{mergingcor}
Let $T, \mathscr J_T^1, \mathscr J_T^2$ be as in Proposition \ref{lem_merging1}. Then for any $\eps>0$, there exists a finite constant $C(\eps)$ such that
\eqref{mergingstatement} also holds for $\mathscr{C}^{\mathscr J_T^1, \mathscr J_T^2}$ in place of $\mathscr{C}_k^{\mathscr J_T^1 \mathscr J_T^2}$.
\end{cor}

\section{Click times on the Ancestral Selection Graph}\label{renASG}
In this section we will define for each $N \in \mathbb N$ a process of click times along the ASG back from some large time $u_N$. The main result of the section will be Proposition~\ref{thm48}, whose proof will build on results in Sections \ref{PercASG},  \ref{processZ} and~\ref{Qickmerge}. Roughly stated, this proposition says that the process of click times on the ASG, back from times that are large on the $f(N)\theta_N$-scale, converges on that scale locally around time $0$ to a standard Poisson process.  This result is key for the proof of Theorem \ref{thmclicks}. Indeed, in Section~\ref{secclickrates} we will argue that  the process of (forward) click times figuring in Theorem \ref{thmclicks}, which are represented as the jump times of the counting process $K_N^\ast$ defined in \eqref{besttype}, is locally on the $f(N)\theta_N$-scale with high probability (as $N\to \infty$) close to the process figuring in Proposition~\ref{thm48}. This latter process, however, can be read off from the ASG decorated with the points of $\mathcal M$. See  Figure~\ref{fig:boat3} and also Figure \ref{fig:boat4} for illustrations. 
\begin{defi}[Backward click times]\label{backwardclicks}
For $N\in \mathbb N$, $u\in \mathbb R$ and $\ell=0,1,\ldots$ we define (again partially suppressing $N$ in the notation) the {\em click times on the ASG back from  $[N]\times \{u\}$} as follows
\begin{equation*}
\hat T_{\ell}^{N,u}:= \min \left\{ t\le u: d_\mathcal{M}({\mathscr{A}}_t^{u},[N]\times\{u\}) = \ell \right\}.
\end{equation*}
We thus get a point process 
$$ \mathscr{T}^{N,u}:= \left\{ \hat T^{N,u}_{\ell}, \ell \in \N_0 \right\}. $$
\end{defi}
For later reference we will consider a sequence  $(u_N)$ of time points with the property
\begin{equation}\label{largetimepoints}
\frac{u_N}{\theta_Nf(N)}\to \infty \quad \mbox{as } N\to \infty.
\end{equation}
\begin{prop}\label{thm48} 
Let $(u_N)$ obey \eqref{largetimepoints} and  $(T_g^{N,u_N})_{1\le g \le g_N}$ be the points contained in the set  $\mathscr{T}^{N,u_N}\cap [0,u_N]$ ordered such that $$ 0 < T_1^{N,u_N}<\cdots < T_{g_N}^{N,u_N}\le u_N.$$
Putting $ T_0^{N,u_N} := 0$, we have for $n \in \N$ the following convergence in distribution as $N\to \infty$:
$$\left( \frac{T_{g}^{N,u_N}-T_{g-1}^{N,u_N}}{f(N)\theta_N} \right)_{1\le g\le n} \to  \left( \mathcal{W}_g \right)_{1\le g\le n}, $$
where $\left( \mathcal{W}_g \right)_{g \in\N}$ is a sequence of i.i.d. standard exponential random variables. Consequently, the sequence of processes $\mathscr N^N$, $N \in \N$, defined by
$$\mathscr N^N_\tau := \sum_{g \geq 1}\mathbf{1}_{\{T_{g}^{N,u_N} \le f(N)\theta_N \tau\}}, \quad \tau\ge 0$$
converges as $N\to \infty$ to a standard Poisson counting process.
\end{prop}

A large part of the remainder of this section is devoted to the proof of  Proposition \ref{thm48}.  With $A^T_t$ denoting the cardinality of the set  $\mathscr A^T_t$ of load $k$ potential ancestors (as defined just before Lemma \ref{fromAtoZ}, see also Definition \ref{colouredASG}), for any $N \in \N$ and any $\mathscr P^{(N)}$-stopping time~$T$  we define the $\mathscr P^{(N)}$-stopping time  $S^{N,T}$ 
\begin{equation} \label{def_SN} S^{N,T}:= 
 \sup\big\{t\le T: A^{T}_t(0) = 0\big\}. \end{equation}
In words, among all times at which all the potential ancestral paths of the population that lives at time $T$ carry at least one mutation, the time $S^{N,T}$ is the one which is closest to $T$. Let us also note that for fixed $N$ the distribution of $T-S^{N,T}$ does not depend on the choice of the $\mathscr P^{(N)}$-stopping time~$T$,
cf.~Lemma \ref{fromAtoZ}. A key step in the proof of Proposition \ref{thm48} is provided by
\begin{lem}\label{asexp}
For any sequence of $\mathscr P^{(N)}$-stopping times $T_N$  the sequence
$$ \frac{T_N - S^{N,T_N}}{f(N)\theta_N} $$
converges in law as $N\to \infty$ to an exponential random variable with rate parameter $1$.
\end{lem}
\begin{proof}
The process $Z_0^{(N)}(r) := A^{T_N}_{T_N-r}(0)$, $r \ge 0$,  has the jump rates~\eqref{Z0rates} and starts in~$N$. Lemma \ref{lem_coupling1} shows that the quasi-equilibrium of $Z^{(N)}_0$ builds up within a time of order $f(N)\ln N$ when started in $Z^{(N)}_0(0) =N$.
Let now $(\theta_N)$ be as in Lemma 6.4. Since this $(\theta_N)$ obeys $(2.7)$ and $f(N)$ satisfies  $f(N) = o\left(\frac N{\ln \ln N}\right)$, we conclude from (2.7) that $\ln \theta_N \gg \ln \ln N$, and hence $f(N) \ln N \ll f(N) \theta_N$.
 The asymptotic exponentiality of $T_N - S^{N,T_N}$ with the claimed time scaling thus follows from Lemma \ref{exp_rate}. \end{proof}

Let us now consider a sequence of (deterministic) times $u_N$ as in Proposition \ref{thm48} and recall the definition of $S^{N, T}$ in \eqref{def_SN}. For each fixed $N\in \mathbb N$ define recursively
\begin{eqnarray}\label{defSN}
\begin{split}  S_{0}^{N,u_N}&:= u_N, \\
 S_{\ell}^{N,u_N}&:= S^{N, S_{\ell-1}^{N,u_N}}, \quad \ell = 1, 2, \ldots 
  \end{split}
  \end{eqnarray}
The following corollary is now immediate from Lemma \ref{asexp}.
\begin{cor}\label{asexpcor}
Let $S_{\ell}^{N,u_N}$ be defined by \eqref{defSN}.\\
a) The sequences 
$$ \left(\frac{S_{\ell-1}^{N,u_N}- S_{\ell}^{N,u_N}}{f(N)\theta_N}\right)_{\ell \ge 1} $$
converge as $N\to \infty$ in the sense of finite dimensional distributions to a sequence of i.i.d. standard exponential random variables.
\\
b) Let $C$ be an arbitrary positive constant. The sequence of point processes
$$ \mathscr S^{N,u_N} := \left\{
\frac{S_{\ell}^{N,u_N}}{f(N)\theta_N}: \ell \in \mathbb N_0
\right\},\quad N=1,2,\ldots$$
converges, when restricted to  $[0,C]$ in distribution to a standard Poisson point process restricted to $[0,C]$.
\end{cor}
\begin{proof}[Proof of Proposition \ref{thm48}] From Definition \ref{backwardclicks} we recall the point process  $\mathscr T^{N,u_N}$ of click times on the ASG back from $[N]\times \{u_N\}$. The strategy of the proof will be to compare this process ``locally on the $f(N)\theta_N$-timescale'' to the process $\mathscr  S^{N,u_N}$ which on that scale according to Corollary \ref{asexpcor} is approximately Poisson.

To this purpose we define for each  $N\in \N$ and each time point $t > 0$ 
$$ S^{N,u_N}(t):= \min  \left(\mathscr  S^{N,u_N} \cap [t,\infty)\right), \qquad  \hat T^{N,u_N}(t):= \min \left(\mathscr  T^{N,u_N} \cap [t,\infty)\right), $$
Let $\bar{\mathscr A}^{T}_t$ be the set of minimum load potential ancestors at time $t$ of the population at some (deterministic or random) time $T$, as specified in Definition \ref{colouredASG} .  For abbreviation we put
\begin{equation*} \mathscr B_t^N:= \bar{\mathscr{A}}^{S^{N,u_N}_{}(t)}_t.
\end{equation*}
For any fixed $C > 0$ we abbreviate $t_N:=C f(N)\theta_N$. We will use the following properties (where always $\delta(\varepsilon) \to 0$ as $\varepsilon \to 0$):
\begin{itemize}
\item The process $(\#\bar{\mathscr A}^{u_N}_{u_N-r},\, r \geq 0)$ follows the dynamics of $Z_{\bar k} = (Z_{\bar k(r)}(r))_{r\ge 0}$, where
\begin{equation} \label{bark}\bar k(r) := \bar k_N(r) := \min\left\{k: Z_k^{(N)}(r) > 0\right\}.
%\ \text{and} \ \widetilde k(r) := \widetilde k_N(r) := \min\left\{k: \widetilde Z_k^{(N)}(r) > 0\right\}
\end{equation} 
%$\bar k$ has been defined in \eqref{bark}.
 According to Lemmata \ref{exp_rate} and \ref{lem_fast_ext}
$$ \liminf_{N \to \infty}\P\left(\#\bar{\mathscr A}^{u_N}_{t_N}\geq \frac{1}{\eps} \right) = 1-\delta(\eps). $$
\item A similar reasoning yields
$$ \liminf_{N \to \infty}
\P
\left(\#\,
\mathscr B_{t_N}^N
\geq \frac{1}{\eps} \right) = 1-\delta(\eps). $$
\item Take a sequence $(v_N)$ such that $f(N)\ln N \ll v_N \ll f(N) \theta_N$. Then according to part~$a)$ of Corollary \ref{asexpcor}, 
$$\limsup_{N\to \infty} \P \left([t_N-v_N, t_N] \cap  f(N)\theta_N  \mathscr S^{N,u_N} = \emptyset\right)=1. $$
\item Finally, according to Proposition \ref{lem_merging1} (on the quick merging of load zero ASG's), 
$$\liminf_{N \to \infty} \P \left(\mathscr{C}_0^{\bar{\mathscr A}^{u_N}_{t_N}, \,\mathscr B_{t_N}^N}\geq t_N-v_N \Big|\#\bar{\mathscr A}^{u_N}_{t_N} \geq 1/\eps \mbox{ and }\#\mathscr B_{t_N}^N\geq 1/\eps \right)\ge 1-\delta(\eps).
$$\end{itemize}
From these facts we deduce that
$$ \lim_{N \to \infty} \P \left( \hat T^{N,u_N}(t_N)=S^{N,u_N}(t_N) \right)=1. $$
We proceed in a similar way to cover the time frame $[0,t_N]$, which contains a random number of points of $\mathscr S^{N,u_N}$ that has a finite expectation. We thus add a sum  of errors that converges to $0$ as $N\to \infty$, which allows us to conclude the proof.
\end{proof}

\section{Click rates: Proof of Theorem \ref{thmclicks}}\label{secclickrates}
The next lemma relates the click times of the ratchet, defined as the jump times of the process $K^\ast_N$ given by \eqref{besttype}, to the times $T_{g}^{N,u_N}$ obtained from the point process $\mathscr{T}^{N,u_N}$ of backward click times, see Definition \ref{backwardclicks} and Proposition \ref{thm48}. As will become clear from the following proof, each time  $T_{g}^{N,u_N}$ with high probability 'announces' a click time of the ratchet, with the difference between those two times tending to zero in probability as $N\to \infty$ on the $f(N)\theta_N$-scale.
\begin{lem}\label{lem10_1}
For any $\tau\ge 0$,
\begin{equation}\label{couplingbackforw}
 \P \left( K^*_N\left(f(N)\theta_N\tau\right)=
\sum_{g \geq 1}\mathbf{1}_{\{T_{g}^{N,u_N} \le f(N)\theta_N \tau\}}  \right) \to 1 \quad \text{as} \quad N \to \infty. 
\end{equation}
\end{lem}
\begin{proof}
Since $K_N^\ast(0)=0$ and $T_g^{N,u_N} > 0$ a.s. for  $g > 0$, we need only to consider the case $\tau > 0$. 
For abbreviation we put $t_N:= f(N)\theta_N^{} \tau$.
Let us recall the definition in~\eqref{min_load_anc} of the set $\bar{\mathscr A}^{[N] \times \{u_N\}}_{t_N}$, abbreviated as $\bar{\mathscr A}^{u_N}_{t_N}$. 
Let $g^*$ be such that 
$$ T_{g^*}^{N,u_N}\leq t_N < T_{g^*+1}^{N,u_N}. $$
Then one readily observes
$$  K^*_N\left(t_N\right) = d_\mathcal{M}\left([N]\times\{0\}, [N]\times\{t_N\} \right) \le d_\mathcal{M}\left([N]\times\{0\}, \bar{\mathscr A}^{u_N}_{t_N} \right) \le g^*. $$
Let us now prove that 
\begin{equation}\label{conversequ}
\P(K_N^\ast(t_N) \ge g^*) \to 1 \quad \mbox{ as } N \to \infty.
\end{equation}
By definition we then have for all $k \in \N_0$ and all $\delta > 0$
\begin{equation}\label{dMestimate}
d_{\mathcal M}\left([N]\times \{T_{g^\ast-k}^{N,u_N}-\delta\},  \bar{\mathscr{A}}^{u_N}_{t_N}\right) \ge k+1.
\end{equation}
Abbreviating $\mathscr J_{t_N}^1 := \bar{\mathscr A}^{u_N}_{t_N}$ and $\mathscr J_{t_N}^2:= [N]\times \{t_N\}$, we observe that $\mathscr J^1_{t_N}$ has the same distribution as $Z_{\underline k_N}^{(N)}(u_N-t_N)$ defined in Lemma \ref{lem_fast_ext}b), which according to that Lemma becomes large with high probability as $N\to \infty$. We thus obtain from Proposition~\ref{lem_merging1} that for all $k \ge 0$ the sets
$$ \mathscr{A}^{\mathscr J_{t_N}^1}(0) \cup ... \cup \mathscr{A}^{\mathscr J_{t_N}^1}(k) \quad \text{and} \quad \mathscr{A}^{\mathscr J_{t_N}^2}(0) \cup ... \cup \mathscr{A}^{\mathscr J_{t_N}^2}(k) $$
merge with high probability within a time of order $f(N) \ln N$. This together with \eqref{dMestimate} allows to conclude that for all $k \ge 0$
$$d_{\mathcal M}([N]\times  \{ T_{g^*-k}^{N,u_N}-\delta\}, [N]\times \{t_N\}) \ge k+1$$
 with high probability as $N \to \infty$.
Applying this to (up to the random) $k:= g^*-1$ shows
$$K_N^*(t_N) = d_{\mathcal M}([N]\times  \{ T_{1}^{N,u_N}-\delta\}, [N]\times \{t_N\}) \ge g^*$$
 with high probability as $N \to \infty$. This concludes the proof of \eqref{conversequ}, and thus also of \eqref{couplingbackforw}. 
\end{proof}
Together with Lemma \ref{lem10_1}, Proposition \ref{thm48}  implies that, as $N\to \infty$, the sequence of processes 
$\displaystyle\left( K^*_N\left(f(N)\theta_N\tau \right)  \right)_{\tau\ge 0}$
converges in distribution to a rate 1 Poisson counting process. This is the assertion of Theorem \ref{thmclicks}.

 \section{Meeting best-type individuals when coming from the future}\label{encounter}
 The following lemma says, roughly spoken, that at generic, suitably large times $t_N$ the minimum load ASG (which is a backward in time construction coming from the far future) with high probability not only is appreciably large but also contains an individual whose type is best among the total population at time $t_N$. This lemma is a building block in the proof of Theorem \ref{thmprofile}a) that will be carried out in Section \ref{YuleinASG}.
\begin{lem}\label{bestinASG2} Let  $(t_N)$ be as in \eqref{condtN} and  $(u_N)$ be such that $(u_N-t_N)/(f(N)\theta_N ) \to \infty$ as $N\to \infty$. Let 
$\bar {\mathscr A}^{u_N}_{t_N}$ be the minimum load ASG at time $t_N$ of the total population at time $u_N$, as specified in Definition~\ref{colouredASG} . Then 
\begin{equation}\label{bestonAbar}
\mathbb P\left(\# \bar {\mathscr A}^{u_N}_{t_N} >h(N); \,   \exists v \in \bar {\mathscr A}^{u_N}_{t_N}: \eta(v) = K^\ast_N(t_N)\right) \to 1 \quad \mbox{ as } N\to \infty,
\end{equation}
where $(h(N))$ is any sequence satisfying $1\ll h(N) \ll N/f(N)$.
\end{lem}
\begin{proof}
Let $\lambda_N$ be such that $ f(N)\ln N \ll \lambda_N \ll f(N)\theta_N \wedge t_N$ and $t_N-\lambda_N \ll f(N) \theta_N$.
From Lemma \ref{fromAtoZ} we know that the law of the process $\big( \# \bar{\mathscr{A}}^{u_N}_{u_N-r}\big)_{r \geq 0}$ is the same as the law of 
the process  $\big(Z^{(N)}_{\bar k(r)}(r)\big)_{r \geq 0}$ studied in Section \ref{processZ} (defined in \eqref{bark}), with initial state $(N,0,...,0,...)$. Thus from  Lemma \ref{lem_fast_ext}b)  we obtain that the distribution of $\#\bar{\mathscr{A}}^{u_N}_{u_N-s}$ at a given time $s \gg f(N)\theta_N$ is  of order $N/f(N)$ with a probability close to one. We may thus apply Corollary \ref{mergingcor} with $\mathscr J_{t_N}^1 :=\bar{\mathscr{A}}^{u_N}_{t_N}$ and $\mathscr J_{t_N}^2 := [N]\times\{t_N\}$ to obtain 
\begin{equation}\label{mergingstatement2} \lim_{N \to \infty} \P \left(\mathscr{C}^{\bar{\mathscr{A}}^{u_N}_{t_N}, [N]\times\{t_N\}}\geq t_N-\lambda_N \right)=1. 
\end{equation}
Take an individual $(i,t_N)$ belonging to the best class at time $t_N$, that is to say 
\begin{equation}\label{eq1} \eta(i,t_N)=K_N^\ast(t_N)=:g^*.
\end{equation}
On the event $\mathscr E_N^{(1)}:=\{\mathscr{C}^{\bar{\mathscr{A}}^{u_N}_{t_N}, [N]\times\{t_N\}}\geq t_N-\lambda_N\}$ we have
$$\mathscr A^{i,t_N}_{t_N-\lambda_N} \subset \mathscr A^{\bar{\mathscr{A}}^{u_N}_{t_N}}_{t_N-\lambda_N}=: \mathscr A_N,$$
and consequently also 
\begin{equation*}
 \eta(i,t_N) \ge \min\left\{\eta(v): v\in \mathscr A^{i,t_N}_{t_N-\lambda_N} \right\} \ge \min\left\{\eta(v): v\in \mathscr A_N \right\}= d_{\mathcal M}\left( \mathscr A_0^{\mathscr A_N}, \mathscr A_N\right).
\end{equation*} 
Now consider the event that there is no click on the ASG between times $t_N-\lambda_N$ and $t_N$, i.e.
$$ \mathscr E_N^{(2)}:= \left\{T_{g^*}^{N,u_N}\leq t_N-\lambda_N \leq t_N < T_{g^*+1}^{N,u_N}\right\}. $$
On this event we have
\begin{equation}\label{eq3}d_{\mathcal M}\left( \mathscr A_0^{\mathscr A_N}, \mathscr A_N\right) =d_\mathcal{M}\left( \mathscr{A}^{ \bar{\mathscr{A}}^{u_N}_{t_N}}_0,  \bar{\mathscr{A}}^{u_N}_{t_N} \right) = \min\left\{\eta(v): v \in \bar{\mathscr{A}}^{u_N}_{t_N}\right\} \ge  K_N^\ast(t_N),
\end{equation}
where the last equality and the last inequality hold by definition.
The chain of \mbox{(in-)}equalities \eqref{eq1}--\eqref{eq3} shows that on the event $\mathscr E_N^{(1)} \cap \mathscr E_N^{(2)}$
$$ \min\left\{\eta(v): v \in \bar{\mathscr{A}}^{u_N}_{t_N}\right\} =  K_N^\ast(t_N).$$
 Proposition \ref{thm48} and \eqref{mergingstatement2} ensure that $\mathbb P(\mathscr E_N^{(1)} \cap \mathscr E_N^{(2)}) \to 1$ as $N\to \infty$, which ends the proof.
 \end{proof}
The next result says that a suitably large set of individuals that live at a  time $\underline t_N \gg f(N)\ln N$, contains with high probability an individual that is of the best type among all individuals living at time $\underline t_N$ (or in other words, has type $K_N^\ast(\underline t_N)$).
 \begin{prop}\label{bigset}
 Let $(\underline t_N)$ obey \eqref{condtN}. For $\eps > 0$ let $\mathscr J^{(N)}_{\underline t_N}$ be a sequence of $\mathscr P_{\underline t_N}$-measurable subsets of $[N]\times\{\underline t_N\}$ with $\lim\limits_{N\to \infty} \P(\#\mathscr J^{(N)}_{\underline t_N} \ge 1/\eps) =1$. Then
 $$\liminf_{N\to \infty} \P\left(\min\left\{\eta^{(N)}(v):v\in \mathscr J^{(N)}_{\underline t_N}\right\}=K^\ast_N(\underline t_N)\right) = 1-\delta(\eps),$$
 with $\delta(\eps) \to 0$ as $\eps\to 0$.
 \end{prop}
\begin{proof}Let $(u_N)$ be such that $(u_N-\underline t_N)/(f(N)\theta_N ) \to \infty$ as $N\to \infty$. Recall the notation of the merging time of two sets in \eqref{def_merg_ancgraph0} and introduce for brevity the notation
\vspace{-0.2cm}
$$ \mathscr{T}_N:=\mathscr{C}_0^{\mathscr J^{(N)}_{\underline t_N},\,\bar {\mathscr A}^{u_N}_{\underline t_N}}$$
for the merging time of zero-load ASG's of $\mathscr J^{(N)}_{\underline t_N}$ and $ \bar {\mathscr A}^{u_N}_{\underline t_N}$.
Since Lemma~\ref{bestinASG2} guarantees that $\bar {\mathscr A}^{u_N}_{\underline t_N}$ is sufficiently large with high probability, an application of Proposition~\ref{lem_merging1} (to the times $\underline t_N$ in place of $T$) yields
$$\liminf_{N \to \infty} \P \left( \mathscr{T}_N \geq \underline t_N-C(\eps) f(N) \ln N \right)\ge 1-\delta(\eps).$$
On the event  that there is no click on $\mathscr A^{u_N}$ between times $\underline t_N$  and $\mathscr T_N$, the minimum load ASG $ \bar {\mathscr A}^{u_N}$ does not acquire additional mutations between those times, hence we have on that event the equality 
\vspace{-0.3cm}
$$
\mathscr A^{{\bar{\mathscr A}}^{u_N}_{\underline t_N}}_{\mathscr T_N}(0) = 
{\bar {\mathscr A}}^{u_N}_{\mathscr T_N}.
$$
One more application of Lemma \ref{bestinASG2}, now to the times $\mathscr T_N$ in place of $\underline t_N$, implies
$$ \mathbb P\left(\exists v \in \bar {\mathscr A}^{u_N}_{\mathscr{T}_N}: \eta^{(N)}(v) = K^\ast_N(\mathscr{T}_N)\right) \to 1 \quad \mbox{ as } N\to \infty. $$
Because of the definition of $\mathscr T_N$, the individual $v$ is a load zero potential ancestor of some  $v^\ast \in \mathscr J^{(N)}_{\underline t_N}$. Consequently, with probability tending to 1 as $N\to \infty$, 
$$ K^\ast_N(\mathscr{T}_N)=\eta^{(N)}(v)=\eta^{(N)}(v^\ast)   =  K^\ast_N(\underline t_N),
$$
which  ends the proof.
\end{proof}

\section{First passage percolation in Poisson-decorated Yule trees}\label{FPinYule}
In this section we consider a Yule tree $\mathscr Y$ with splitting rate $\alpha$, and regard $\mathscr Y$ as the union of the (infinitely many) {\em  lineages} $\mathfrak l$ leading from the root to $\infty$. More formally, we regard a realisation of $\mathscr Y$ as the union $$\mathfrak y= \bigcup\limits_{\iota \in \mathscr U} \{\iota\}\times [\tau_\iota, \infty),$$ where $$\mathscr U = \{\emptyset\} \cup \bigcup_{g\in \mathbb N} \mathbb N^g$$ is the Ulam-Harris index set (meaning that the  branch with index \mbox{$\iota_1\ldots \iota_g\in\mathscr U$} is the $\iota_g$-th branch born by the branch with index $\iota_1\ldots \iota_{g-1}$), and $\tau_\iota$ is the birth time of the branch with index $\iota$. We think of $(\iota, h)$ as a {\em node} in the tree $\mathfrak y$, and refer to $h$ as its {\em heigth}. \\ Given $\mathscr Y$, let $\Pi$ be a Poisson process on $\mathscr Y$ whose intensity is $\mu$ times the length measure on  $\mathscr Y$. (In Section \ref{YuleinASG} we will prove that these {\em Poisson-decorated Yule trees} indeed appear  in the ASG as $N\to \infty$, see Figure \ref {fig:boat4} for an illustration).
Again we assume $\mu < \alpha$ and define the {\em minimal $\Pi$-load in $\mathscr Y$} as
\begin{equation}\label{defK}
L:= \min\{\Pi(\mathfrak l) : \mathfrak l  \mbox{ is a  lineage of } \mathscr Y\}.
\end{equation}
In this section we will use the abbreviation 
$$q:= \frac \mu{\alpha + \mu}, \quad \rho :=  \frac \mu \alpha = \frac q{1-q}. $$  
\begin{prop}\label{pi_eqals_p} Let the random variable $L$ be defined by \eqref{defK}.\\
a) Recall the Definition of $\mathfrak G$ in \eqref{defG}. Then there exists some constant $C_\rho > 0$ such that
\begin{equation}\label{tailsofK}
 \P(L>\ell) =   \underbrace{\mathfrak{G}\circ... \circ\mathfrak{G}}_{\ell  \text{ times}}(\rho), \qquad \ell \in \N_0,
 \end{equation}
\begin{equation}\label{geometrictails}
\P(L> \ell) \sim C_\rho \left(\frac \rho{1+\rho}\right)^\ell \quad \mbox{ as }\,  \ell \to \infty.
\end{equation}
b) The probability weights of $L$ $(\pi_k)_{k \in \N_0} := (\P(L=k))_{k \in \N_0}$ satisfy the recursion \eqref{recursion} with $\pi_0 = 1-\rho$,  and thus are equal to the weights $(p_k)_{k\in \N_0}$, appearing in Theorem \ref{thmprofile}. 
\end{prop}
\begin{proof}
1.  The Yule tree $\mathscr Y$ together with its Poisson decoration $\Pi$ (that enter in the definition of $L$ in \eqref{defK}) define a binary branching supercritical  Galton-Watson tree $\mathscr G$ as follows:  when moving away from the root of $\mathscr Y$, every first encouter with a point of $\Pi$ stands for a death in $\mathscr G$, while every splitting point of  $\mathscr Y$ stands for a birth in $\mathscr G$. Hence $\mathscr G$ has offspring distribution
\begin{eqnarray*} 
\P\left(\mbox{no child}\right) = q  =1-
 \P\left(\mbox{two children}\right).
 \end{eqnarray*}
The event that $\mathscr G$ is finite equals the event that there is no lineage $\mathfrak l$ in $\mathscr Y$ with $\Pi(\mathfrak l) = 0$, which in turn equals the event $\{L>0\}$. 
A first step decomposition shows that the extinction probability of $\mathscr G$ is $q/({1-q}) = \rho$ (cf. also \eqref{firstgen} and \eqref{gequalsG2} below), hence
\begin{equation}
\label{Kq}
1-\pi_0 =  \P(L>0) = \P(\#\mathscr G < \infty) = \rho.
 \end{equation}
%which is \eqref{tailsofK} in the special case $\ell = 0$.  
2. Exploring the lineages of $\mathscr Y$ beyond the points of $\Pi$ that are closest to the root of $\mathscr Y$, we encounter a self-similar situation: any such point can be seen as the root of an independent copy of $\mathscr G$, and the event $\{L>1\}$ equals the event  that all of these Galton-Watson trees are finite, which in view of \eqref{Kq} has probability
\begin{equation}\label{Lbiggerone}
\P(L>1) =\E[\rho^{\mathfrak m_1}I_{\{\#\mathscr G < \infty\}}]
\end{equation}
where $\mathfrak m_1$ is the number of leaves of $\mathscr G$. We put
 \begin{eqnarray}\label{gequalsG1}
 g(u) := \E[u^{\mathfrak m_1}I_{\{\#\mathscr G < \infty\}}] = \begin{cases} \E[u^{\mathfrak m_1}] , & 0\le u < 1\\ \P\left(\#\mathscr G < \infty\right), \quad  &u = 1. \end{cases}
 \end{eqnarray}
 A first generation decomposition gives
\begin{equation} \label{firstgen}
g(u) = q u+ (1-q) g(u)^2.
\end{equation}
From the two solutions of this equation only the function $\mathfrak G$ given by \eqref{defG} 
is admissible, since from \eqref{gequalsG1} and \eqref{Kq} we  have that $g(1) = \rho < 1$. 
Consequently, we have
\begin{equation}\label{gequalsG2}
\E[u^{\mathfrak m_1}I_{\{\#\mathscr G < \infty\}}] = \mathfrak G(u), \quad 0 \le u \le 1.
 \end{equation}
 Combining \eqref{Lbiggerone} and \eqref{gequalsG2} for $u:=\rho$ gives \eqref{tailsofK} for $\ell = 1$. Proceeding further, $\{L>2\}$ means that all of the $\mathfrak m_1$ many Poisson points are founders of lineages that carry more than one point of $\Pi$. This event has probability
$$\P(L>2) = \E[  (\mathfrak G(\rho))^{\mathfrak m_1}] = \mathfrak G(\mathfrak G(\rho)),$$
which is \eqref{tailsofK} for $\ell =2$. For general $\ell \in \N$, formula \eqref{tailsofK} follows by induction.
\\
3. From \eqref{defG} we have 
$$\mathfrak G'(s) = \rho\left((1+\rho)^2-4\rho s\right)^{-1/2},\quad 0\le s \le 1$$
and hence 
\begin{equation}\label{Gprimezero}
0 < q= \mathfrak G'(0) \le \mathfrak G'(s) \le \mathfrak G'(1) = \rho < 1,\quad 0\le s \le 1.
\end{equation}
For $\ell \in \mathbb N_0$ let $a_\ell$ be defined by the right hand side of \eqref{tailsofK}. Since $\mathfrak G$ is twice continuoulsly differentiable with $\mathfrak G(0)=0$ and $\mathfrak G'(0)=q$,  we have for some $0 < \chi_\ell < \xi_\ell < a_\ell$
$$a_0 = \rho, \qquad a_{\ell+1}=\mathfrak G(a_\ell) =  a_\ell \mathfrak G'(\xi_\ell) = a_\ell(q+\mathfrak G''(\chi_\ell)\xi_\ell).$$
Hence with $R_\ell := \mathfrak G''(\chi_\ell)\xi_\ell$ 
we obtain 
$$\frac{a_\ell}{q^\ell} = \rho\prod_{j=0}^{\ell-1}\left(1 + \frac {R_\ell}{q}\right).$$
Since $\mathfrak G$ is Lipschitz continuous on $[0,1]$ with Lipschitz constant $\rho < 1$ and \mbox{$\mathfrak G(0)=0$, }we have $a_\ell \le \rho^{\ell+1}$. From the fact that $\mathfrak G''$ is nonnegative and increasing we deduce that
$ 0\le R_\ell\le \mathfrak G''(1) \rho^\ell$, showing that $a_\ell \sim q^\ell  C_\rho$ with $C_\rho = \rho\prod_{j=0}^\infty\left(1 + \frac {R_\ell}{q}\right) $. In view of \eqref{tailsofK}  this proves that the random variable $L$ has the tail asymptotics
\eqref{geometrictails}.\\
4. Let $\mathfrak e$ be the edge  that is between the root of  $\mathscr Y$ and its closest branch point. The random variable 
$M:= \Pi(\mathfrak e)$ satisfies
\begin{equation}\label{geom}
\P(M \geq \ell) = q^\ell, \quad \ell \in \N_0.
\end{equation}
The random variable $L$ satisfies the stochastic fixed point equation
\begin{equation}\label{fpe}
L \stackrel d= M+ \min(L_1,L_2)
\end{equation}
where $L_1$ and $L_2$ have the same distribution as $L$ and $L,L_1,L_2, M$ are independent. Hence
\begin{eqnarray}\label{distweight}
\P(M+ \min(L_1,L_2) = k) = \sum_{i=0}^k \P(M=k-i) \P(\min (L_1,L_2) = i), \quad k\in \N_0.
\end{eqnarray}
From the independence of $L_1$, $L_2$ we have
$$ \P(\min (L_1,L_2) = i) = \pi_{i}^2 + 2\pi_{i}\sum_{j>i}\pi_{j}=:w_i.$$
From \eqref{geom} we have
$$\P(M=k-i)=  q^{k-i}(1- q).$$
Inserting this into \eqref{distweight} and observing \eqref{fpe} we obtain
$$\pi_k = \sum_{i=0}^k  q^{k-i}(1- q)w_i, \quad k \in \N_0.$$
Taking differences yields
$$\pi_k-\pi_{k-1} = (1- q)w_k-\sum_{i=0}^{k-1} q^{k-1-i}(1- q)^2w_i = (1- q)(w_k -\pi_{k-1}). $$
 Observing that $(1- q) (1+\rho)=1$  we arrive at
\begin{eqnarray}
(1+\rho)(\pi_k-\pi_{k-1}) =  \pi_{k}\left(\pi_k + 2\sum_{j>k}\pi_{j}\right)-\pi_{k-1}
\end{eqnarray}
which is equivalent to
\vspace{-0.2cm}
\begin{equation}\label{newrec}
\rho\, (\pi_k-\pi_{k-1})=  \pi_{k}\left(\pi_k -1+ 2\sum_{j>k}\pi_{j}\right).
\end{equation}
\vspace{-0.2cm}
In view of \eqref{geometrictails} we have $\P(L<\infty)=1$; hence
$$\pi_k-1 +2\sum_{j>k}\pi_j= -\pi_k+1 -2+2\pi_k+2\sum_{j>k}\pi_j = -\pi_k+1-2\sum_{j<k}\pi_j.$$
Thus \eqref{newrec} together with \eqref{Kq}  shows that ($\pi_k)$ satisfies the recursion \eqref{recursion}.
\end{proof}
\begin{rem} \label{geome}
The setting of Proposition \ref{pi_eqals_p}  gives an instance of Example 40 in \cite{AB05}: our stochastic fixed point Equation~\eqref{fpe} corresponds to Eq.~(49) in~\cite{AB05} with a geometrically distributed ``toll'' random variable $\eta$. Thus, the results of Proposition \ref{pi_eqals_p}  apply to a specific case of a situation which, according to \cite{AB05}, ``does not seem to have been studied generally''. As stated in Theorem~\ref{thmprofile}d) and explained in Section~\ref{proofbtod}, this connects to the asymptotic minimum of a branching random walk whose increment distribution is supported on $\R_+$.
(See \cite{hu2016big} and references therein for the asymptotics of minima of random walks with two-sided increment distributions.)
\end{rem}
In order to prepare for the connection between Proposition \ref{pi_eqals_p} and the decorated ASG of a sampled individual, we need some more notation. 
\begin{defi}
Let $\mathscr Y$ be the Yule tree described at the beginning of the section. For a node $v \in \mathscr Y$, let $\mathfrak a(v)$ be the path from $v$ to the root,  and for $h>0$, $k \in \N_0$ let $\mathscr Y_h(k)$ be the set of nodes $v$ of~$\mathscr Y$ that have height $h$ and obey $\Pi(\mathfrak a(v))=k$.  
Finally, we define, as an analogue to \eqref{defK}, the {\em minimal $\Pi$-load in $\mathscr Y$ up to height $h$} as
 \begin{equation}\label{defKt}
 L_h := \min\{\Pi(\mathfrak a(v)):v \in \mathscr Y_h\}.
 \end{equation}
 \end{defi}
 The following lemma says that the minimal $\Pi$-load of the (infinite) lineages in $\mathscr Y$ can with high probability be observed already at a  height $w_N$ which is large but of smaller order than $\ln N$ as $N \to \infty$; moreover at this height there are many nodes of the Yule tree whose ancestral paths collect this load.
\begin{lem} \label{Yuletruncate}
Let $w_N \to \infty$ with $w_N=o(\ln N)$ as $N\to \infty$. Then for all $k \in \N_0$,
\begin{equation}\label{truncate}
\P\left(L_{w_N} = k, \, e^{(\alpha-\mu)w_N/2} \leq \# \mathscr Y_{w_N}(k) \leq e^{2(\alpha-\mu)w_N} \right) \to p_k \mbox{ as } N\to \infty.
\end{equation}
\end{lem}
\begin{proof}
As described in the proof of Proposition \ref{pi_eqals_p}, an equivalent representation of a binary branching Galton-Watson tree with mutation at rate $\mu$ is a sequence of trees of different types killed at rate $\mu$. Descendants of the root are of type $0$. Every death of an individual of type $0$ leads to a new binary branching Galton-Watson tree of type $1$ and so on. Let us consider the event $\{L_{w_N} = k\}$. This event implies that all the trees of types $l \leq k-1$ are extinct at time~$w_N$. Every such tree is a supercritical tree with birth rate $ \alpha$ and death rate~$\mu$. Conditioned on extinction, it is thus a subcritical tree with birth rate $\mu$ and death rate $ \alpha$, and it has a mean number of leaves 
$$\mathfrak{G}'(1)=\tfrac {\rho}{1-\rho} = \tfrac \mu{\alpha-\mu}$$
 (cf. \eqref{Gprimezero}) and a finite mean extinction time. Hence 
$$ \lim_{N \to \infty}\P \left( \# \mathscr Y_{w_N^{1/2}}(l)=0,0 \leq l \leq k-1| L_{w_N} = k\right)=1. $$
Moreover, by definition, any tree of type $k$ still alive at time $w_N$ is born before the death of the last alive type $k-1$ individual. 
Let us denote by $\widehat{\mathscr Y}$ a binary Galton-Watson tree of type~$k$ with birth rate $\alpha$ and death rate $\mu$. On the event of survival (see for instance \cite{athreya1972branching} p.112), 
$$ \lim_{t \to \infty} (\ln \widehat{\mathscr Y}_t)/t = \alpha-\mu. $$
On the event
$$ \mathfrak{E}_{k,N}:=\left\{\# \mathscr Y_{w_N^{1/2}}(l)=0,0 \leq l \leq k-1\right\} \cap \{ L_{w_N} = k \}$$
we know that 
\begin{itemize}
\item There is a finite mean number of independent copies of $\widehat{\mathscr Y}$ and a positive number of them survive after time $w_N$ which goes to infinity with $N$.
\item These independent copies have a root born between the times $0$ and $w_N^{1/2}$.
\end{itemize}
We deduce that
\vspace{-0.2cm}
$$ \lim_{N \to \infty}\P \left( e^{(\alpha-\mu)w_N/2} \leq \# \mathscr Y_{w_N}(k) \leq e^{2(\alpha-\mu)w_N}| \mathfrak{E}_{k,N}\right)=1. $$
Finally, notice that from properties of supercritical Galton-Watson processes,
$$ \lim_{N \to\infty}\P\left( \# \mathscr Y_{\infty}(k) \geq 1 | \# \mathscr Y_{w_N}(k) \ge e^{(\alpha-\mu)w_N/2} \right)=1. $$
We thus obtain
$$ \lim_{N \to\infty}\P\left( L \leq k | L_{w_N}=k,\# \mathscr Y_{w_N}(k) \ge e^{(\alpha-\mu)w_N/2}  \right)=1. $$
But by definition, $L_{w_N} \leq L$, which yields
$$\P\left(L_{w_N} = k, \, e^{(\alpha-\mu)w_N/2} \leq \# \mathscr Y_{w_N}(k) \leq e^{2(\alpha-\mu)w_N} \right) \to \P\left(L = k\right) \mbox{ as } N\to \infty.$$
An application of Proposition \ref{pi_eqals_p} ends the proof.
\end{proof}

\begin{lem}\label{Yuleapprox}
Let $(w_N)$ be as in Lemma \ref{Yuletruncate}, and assume  that the splitting rate $\alpha_N(h)$ and the decoration rate $\mu_N(h)$  may depend on $N$ and $h$ such that, uniformly in $h \in [0, w_N]$,
\begin{equation}\label{asrates}
 \lim_{N \to \infty} \alpha_N(h)=\alpha \mbox{ and }   \lim_{N \to \infty} \mu_N(h)=\mu.
\end{equation}
 Let $\mathscr Y^{(N)}$ be the corresponding $\Pi^{(N)}$-decorated Yule tree, grown up to the height $w_N$. For $0\le h \le w_N$, define in analogy to  \eqref{defKt}
 $$ L^{(N)}_h:= \min\{\Pi^{(N)}(\mathfrak a(v)):v \in \mathscr Y^{(N)}_h\}.$$ 
Then, for all $k \in \N$, 
  $$\pi_k^{(N)} := \P(L_{w_N}^{(N)} =k) \to \pi_k \, \mbox{ as } N\to \infty,$$
  where $(\pi_k)$ is as in Proposition \ref{pi_eqals_p}. Moreover, the analogue of \eqref{truncate} holds for $L^{(N)}$ and $\mathscr Y^{(N)}$ instead of $L$ and $\mathscr Y$.
\end{lem}
\begin{proof}  All the previous quantities (probabilities, mean numbers, growth rates and expected times) are continuous functions of the parameters $\alpha$ and $\mu$. Sandwiching arguments thus allow to extend the proof of the previous lemma. 
\end{proof}

 \section{First passage percolation within the ASG. Proof of Theorem \ref{thmprofile}a)}\label{YuleinASG}
In this section we will complete the proof of Theorem \ref{thmprofile}a) along the program laid out in Figure~\ref{fig:boat4}. This program has two parts. The first one says that the $\mathcal M$-decorated ASG's of single individuals look at the time scale~$f(N)$ like the Poisson-decorated Yule processes studied in Section \ref{FPinYule}. Building on the results of Section \ref{FPinYule}, we will prove this in Lemma~\ref{lem_indep}, together with the fact that on this time scale (and slightly beyond it) the ASG's of finitely many individuals are asymptotically independent. Roughly stated, Lemma~\ref{lem_indep} says that (with $w_N$ as in Lemma \ref{Yuletruncate})  the minimal-load potential ancestors at time $t_N-w_Nf(N)$ of an individual $(i,t_N)$ are numerous and that the minimal load $L_{w_Nf(N)}$ acquired over the time span $w_Nf(N)$ has asymptotically as $N\to \infty$ the distribution $(p_k)_{k \in \N_0}$ given by the recursion \eqref{recursion}, which we encoutered also in Proposition~\ref{pi_eqals_p}.
 The second part of the program announced in Figure \ref{fig:boat4} can be stated as the fact that a suitably large set of individuals that live at a  time $\underline t_N \gg f(N)\ln N$, contains with high probability an individual that is of the best type among all individuals living at time $\underline t_N$ (or in other words, has type $K_N^\ast(\underline t_N)$). This was proved in Proposition \ref{bigset}, and will  be applied to the set of minimum-load potential ancestors at time $\underline t_N:= t_N-w_Nf(N)$ of an individual $(i,t_N)$, showing that  this individual's type $\eta^{(N)}(i,t_N)$ is indeed with high probability the sum of the best type $K_N^\ast(t_N)$ in the total population at time $t_N$  and the minimal load $L_{w_N f(N)}$ acquired by the potential ancestry of $(i,t_N)$ between times $\underline t_N$ and $t_N$.

For the next lemma let us define, as an analogue to \eqref{defKt}, the {\em minimal load of potential ancestors at time  $t-r$ of an individual $(i,t) \in G^{(N)}$} as
 \begin{equation}\label{defKt1}
 L_r^{i,t} := d_{\mathcal M}\left( \mathscr{A}_{t-r}^{i,t}\, , (i,t)\right). 
 \end{equation}
\begin{lem}\label{lem_indep}
Let $t_N$ be as in \eqref{condtN} and $w_N$ be as in Lemma \ref{Yuletruncate}.
Choose a sample size $n\in \N$, let $i_1, \ldots, i_n \in \mathbb N$ be pairwise distinct, and $k_1,\ldots, k_n \in \N_0$. Then, for all $\eps > 0$,
$$\lim_{N\to \infty} \P\left( L_{f(N)w_N}^{i_\ell,t_N}=k_\ell, \, A_{t_N-f(N)w_N}^{i_\ell,t_N}(k_\ell)\geq 1/\eps, 1 \leq \ell \leq n \right)= p_{k_1}\cdots p_{k_n}.$$
\end{lem}
\begin{proof}
First notice that if we take independent Yule trees $(\mathscr Y^{(\ell)}, 1 \leq \ell \leq n)$ as in Lemma \ref{Yuletruncate} and denote their minimal $\Pi$-load accordingly, we get from an application of Lemma \ref{Yuletruncate},
$$ \lim_{N\to \infty}\P\left(L_{w_N}^{(\ell)} = k_\ell, \,  e^{(\alpha-\mu)w_N/2} \leq \# \mathscr Y^{(\ell)}_{w_N}(k_\ell) \leq e^{2(\alpha-\mu)w_N}, 1 \leq \ell \leq n  \right) = p_{k_1}\cdots p_{k_n}.$$
According to Lemma \ref{Yuleapprox}, this result still holds true if the splitting rate $\alpha_N(h)$ and the decoration rate $\mu_N(h)$ of the Yule trees $(\mathscr Y^{(\ell)},1\leq \ell \leq n)$  may depend on $N$ and $h$ in such a way that \eqref{asrates} is fulfilled. 
Hence, to prove the lemma we need to show two properties on the processes $(A_{t_N-r}^{i_\ell,t_N}, r \geq 0)$
\begin{itemize}
\item The rates of the processes $\left(A_{t_N-hf(N)}^{i_\ell,t_N}, h \geq 0\right)$ indeed follow \eqref{asrates} 
\item They are asymptotically independent in the following sense
$$ \lim_{N \to \infty} \P \left( \bigcap_{1 \leq \ell \leq n}\mathscr A_{t_N-r}^{i_\ell,t_N}=\emptyset, \forall r \leq f(N)w_N \right)=1. $$
\end{itemize}
Recall the rates of the processes $(A_{t_N-r}^{i_\ell,t_N}, r \geq 0)$ as stated in Lemma \ref{fromAtoZ}. In particular, they imply that these processes are dominated by a Yule process with birth rate $\alpha/f(N)$. 
Hence, if we introduce the event
$$ \mathcal{E}^{(N,{\rm not big})}:=\left\{\sup_{r \leq f(N)w_N,1 \leq \ell \leq n}A_{t_N-r}^{i_\ell,t_N} \leq e^{2 \alpha w_N}\right\}, $$
and apply Lemma A.1 in~\cite{chazottes2016sharp}, we get
\begin{equation} \label{min_merge} \lim_{N \to \infty} \P \left(  \mathcal{E}^{(N,{\rm not big})} \right)=1. \end{equation}
The rate at which two currently disjoint ASG's of sizes $a_1$ and $a_2$ acquire a common point is $\frac {a_1 a_2}N$, cf. \eqref{r3}.  On the event $\mathcal{E}^{(N,{\rm not big})}$, the mean number of such events during the time $[0,f(N)w_N]$ for two processes among $(\mathscr A_{t_N-r}^{i_\ell,t}, 1 \leq \ell \leq n)$ is thus bounded by
$$ \frac{1}{2N} (2e^{2 \alpha w_N})(2e^{2 \alpha w_N}-1)w_N,$$
which converges to 0 as $N\to \infty$. Applying the Markov inequality to the number of such events and using \eqref{min_merge}  completes the proof.
\end{proof}
\begin{prop}\label{propchaos1} Let the type configurations $\eta^{(N)}(i,t)$, $i\in [N]$, $t\ge 0$, be as  in Section~\ref{GRep}, with $\eta^{(N)}(i,0):= 0$, and let $K^\ast_N(t)$ be as in~\eqref{besttype}. Let $(t_N)$ be a sequence of time points which obeys \eqref{condtN}. Choose a sample size $n\in \N$, let $i_1, \ldots, i_n \in \mathbb N$ be pairwise distinct, and $k_1,\ldots, k_n \in \N_0$. Then, with $(p_k)_{k \in \N_0}$ given by the recursion \eqref{recursion},
$$\lim_{N\to \infty} \P\left(\eta^{(N)}(i_1,t_N)-K_N^\ast(t_N)=k_1, \ldots, \eta^{(N)}(i_n,t_N)-K_N^\ast(t_N)=k_n\right)= p_{k_1}\cdots p_{k_n}.$$
\end{prop}
\begin{proof} Let $w_N$ be as in Lemma \ref{lem_indep}, $\underline t_N:= t_N-w_Nf(N)$, and $(u_N)$ satisfy $(u_N-t_N)/(\theta_N f(N)) \to \infty$ as $N\to \infty$. Applying Lemma \ref{bigset} to the set $\mathscr J^{(N)}_{\underline t_N}:={\bar {\mathscr A}}^{i_\ell,t_N}_{\underline t_N}$  of minimum load potential ancestors at time~$\underline t_N$ of the individual $(i_\ell,t_N)$ (see Definition~\ref{colouredASG}) yields
$$\P\left(\min\{\eta^{(N)}(v): v\in {\bar {\mathscr A}}^{i,t_N}_{\underline t_N} \} = K_N^\ast(\underline t_N)\right) \to 1 \quad \mbox{as  } N\to \infty.$$
From Theorem \ref{thmclicks} we know that
$$\P(K_N^\ast(t_N) = K_N^\ast(\underline t_N)) \to 1 \quad \mbox{as } N\to \infty.$$
Hence we conclude, using the notation \eqref{defKt1}, 
$$\P\left(\eta^{(N)}(i_\ell, t_N) = K_N^\ast(t_N)+ L^{i_\ell, t_N}_{w_Nf(N)} \right) \to 1 \quad \mbox{as } N\to \infty.$$
An application of Lemma \ref{lem_indep} thus concludes the proof.
\end{proof} 
In accordance with the graphical representations of $\xi^{(N)}(t)$ and $K_N^\ast(t)$ in \eqref{defXk1} and \eqref{repKN}, the empirical type frequency profile seen from the currently best type (defined in  \eqref{defxi}) is represented as
$$X_k^{(N)}(t) = \sum_{i=1}^N \delta_{\eta^{(N)}(i,t)-K_N^\ast(t)}(k), \quad k \in \N_0,\, t\ge 0.$$
The following corollary concludes our  proof of Theorem \ref{thmprofile} a).
\begin{cor} For $(t_N)$  obeying \eqref{condtN}, and all $k \in \mathbb N_0$,  
 $X_k^{(N)}(t_N)$ converges  in probability \mbox{as $N\to \infty$}  to $p_k$, with
\mbox{$(p_k)_{k \in \N_0}$} given by \eqref{recursion}.
\end{cor}
\begin{proof}For all $k \in \mathbb N_0$, the second moment $\mathbb E[X_k^{(N)}(t_N)^2]$ is asymptotically equal to the probability that, for  $J_1, J_2$  randomly sampled from $[N]$, the types  $\eta^{(N)}(J_1,t_N)-K_N^\ast(t_N)$ and $\eta^{(N)}(J_2,t_N)-K_N^\ast(t_N)$ both are equal to $k$. Proposition \ref{propchaos1} tells that this probability converges to $p_k^2$ as $N\to \infty$. By the same proposition,  $p_k^2$ is the limit of $\mathbb E[X_k^{(N)}(t_N)]^2$ as $N\to \infty$, hence the variance of $X_k^{(N)}(t_N)$ vanishes as $N\to \infty$.
\end{proof}
\section{The quasi-stationary type frequency profile: Proof of Theorem  \ref{thmprofile}b)-e)}\label{proofbtod}
\paragraph{Part b)}Choosing $k=0$ in \eqref{systeq} and using the assumption that $p_0+p_1+ \cdots = 1$, we see that  \eqref{systeq} implies $p_0=1-\tfrac\mu\alpha$. 
 The equivalence of \eqref{recursion} and  \eqref{systeq}  is then  immediate from the identity 
$$-p_k +1-2\sum_{k'=0}^{k-1}p_{k'} = \sum_{k'=k+1}^\infty p_{k'} - \sum_{k'=0}^{k-1} p_{k'}.$$
\paragraph{Part c)} The characterisation of $(p_k)_{k \in \N_0}$ in terms of minimal Poisson loads of infinite lineages in a Yule tree has been proved in Proposition \ref{pi_eqals_p} b).

To see the equivalence to the characterisation via the eventual minimum in a branching one-sided random walk, we think of the latter  as a family of random walks, indexed by the infinite lineages $\mathfrak l$ of a  Yule tree $\mathscr T$ with branching rate $\alpha$.  All the random walkers move on $\N_0$, starting at $0$ and jumping from $k$ to $k+1$ at rate $\mu$. The (correlated) dynamics of the walkers can thus be seen as driven by a Poisson point process $\Pi$ with rate $\mu$ on $\mathscr T$:
each point of $\Pi$ induces an upwards jump by one, and the (continuous) time of the walks corresponds to the height in $\mathscr T$. Thus the position at time $t$ of the walker that is indexed by an infinite lineage $\mathfrak l$ of $\mathscr T$ is the number of Poisson points which $\mathfrak l$ carries between heights $0$ and $t$. Denoting by $M(t)$ the minimum of the position of all the walkers alive at time $t$, we see that $M(t)$ increases to the $\N_0 \cup \{\infty\}$-valued random variable  $K:= \min\{\Pi(\mathfrak l):\mathfrak l \in \mathscr T\}$, i.e. the minimum over the numbers of Poisson points carried by the infinite lineages in $\mathscr T$.
\paragraph{Part d)}
This is part of the assertion of Proposition \ref{pi_eqals_p}.
\paragraph{Part e)}
Abbreviating $1- \mu/\alpha = \beta$ we get from \eqref{recursion}:
$$p_0 = \beta, \quad p_1= \sqrt{\left(\frac \beta 2\right)^2+\beta(1-\beta)}-\frac \beta 2.$$
Therefore, $p_0^2 \ge p_1^2 \Longleftrightarrow 9\beta^2 \ge \beta^2 + 4\beta(1-\beta) \Longleftrightarrow 2\beta \ge 1-\beta \Longleftrightarrow \beta \le 1/3,$
with equality iff $\beta =  1/3$.
This proves the assertions (i) and (ii). For checking (iii), we define 
$$\mathfrak k := \max\left\{k\in \N_0: \sum_{k'< k}p_{k'}  \le  \sum_{k'> k}p_{k'}\right\}.$$
For $k:= \mathfrak k +1$, the l.h.s. of \eqref{systeq} is strictly negative, hence $p_{\mathfrak k +1} < p_{\mathfrak k}$. Since the l.h.s. of \eqref{systeq} is strictly decreasing in $k$ and thus can be zero for at most one $k$, it must be strictly positive for $k:=  \mathfrak k -1$, hence, again because of \eqref{systeq} we have $p_{\mathfrak k -2} < p_{\mathfrak k -1}$.\\
The claim concerning the geometric tail of $(p_k)_{k \in \N_0}$ follows immediately from \eqref{geometrictails} combined with Proposition \ref{pi_eqals_p}b).
This concludes the proof of Theorem  \ref{thmprofile}.

\section{Proof of Lemmata \ref{lem_coupling1}, \ref{exp_rate} and \ref{lem_fast_ext}}

This section is dedicated to the study of the process $(Z^{(N)}(r), r \geq 0)$. The proof of Lemma \ref{lem_coupling1} relies essentially on the fact that a stochastic Lotka-Volterra process with large carrying capacity $\bar K$ resembles a supercritical process when its size is small and once close to its carrying capacity, stays in a neighboorhoud of the latter during any time of order $\ln \bar K$. This last property is stated in Lemma C.1 in \cite{champagnat2021stochastic}, and will be instrumental in the following proof.
\begin{proof}[Proof of Lemma \ref{lem_coupling1}]\label{ProofofLemmas}
We will prove this result by induction. Let us first consider the case $k=0$ and introduce, for $\eps \in (0,\alpha-\mu)$, the notation
$$ \mathfrak{e}(\alpha,\mu,\eps):= 2(\alpha - \mu - \eps). $$
Notice that for $N$ large enough and $n\leq \mathfrak{e}(\alpha,\mu,\eps)N/f(N)$, the birth and death rates defined in \eqref{cor_rates_01} obey
$$ b_0(n,N) \geq \frac{\alpha}{f(N)} \left( 1- \frac{\eps}{4\alpha} \right) \quad \text{and} \quad d_0(n,N) \leq \frac{\alpha}{f(N)} \left( 1- \frac{\eps}{2\alpha} \right). $$
Thus, if $Z_0(0) \leq \mathfrak{e}(\alpha,\mu,\eps)N/f(N)$, before its hitting time of $\mathfrak{e}(\alpha,\mu,\eps)N/f(N)$, $Z_0$ stochastically dominates a supercritical branching process with growth rate
$$ \frac{\alpha}{f(N)}\frac{\eps}{4\alpha}=\frac{\eps}{4f(N)}$$
and initial state $\lfloor 1/\eps\rfloor$. By \cite[Lemma C.1] {champagnat2021stochastic} the probability that the latter reaches the size $\mathfrak{e}(\alpha,\mu,\eps){N}/{f(N)}$ by time $6f(N) \ln N/\eps$ is close to 1, hence 
$$ \liminf_{N \to \infty} \P\left( \inf \left\{r \geq 0, Z_0(r)\geq\mathfrak{e}(\alpha,\mu,\eps)\frac{N}{f(N)}\right\} \leq \frac{6}{\eps}f(N) \ln N \Big|Z_0(0) \geq 1/\eps\right)=1-\delta(\eps), $$
where $\delta(\eps) \to 0$ as $\eps \to 0$.
A similar comparison, now {\em above} the threshold $\mathfrak{e}(\alpha,\mu,-\eps)$  and with a suitable birth-and-death process with {\em negative} drift, shows that 
$$ \lim_{N \to \infty} \P\left( \inf \left\{r \geq 0, Z_0(r)\leq\mathfrak{e}(\alpha,\mu,-\eps)\frac{N}{f(N)}\right\} \leq \frac{3}{\eps}f(N) \ln N \right)=1. $$
$Z_0$ thus takes a time of order $f(N)\ln N$ to reach the interval $\left[ \mathfrak{e}(\alpha,\mu,\eps),\mathfrak{e}(\alpha,\mu,-\eps)\right] $. We now prove that it then remains in this interval for a longer period on that timescale. 

By definition, the process $Z_0$ cannot exceed $N$. As a consequence, for any $r\geq 0$,
$$ b_0(Z_0(r),N) \leq  \frac{\alpha}{f(N)} \quad \text{and} \quad  d_0(Z_0(r),N)  \geq \frac{\alpha}{f(N)}  \frac{\mu}{\alpha}= \frac{\mu}{f(N)}. $$
Thus, applying Lemma C.1 in \cite{champagnat2021stochastic} (after a time change with the factor $f(N)$), we obtain that for any $R<\infty$, 
\vspace{-0.3cm}
$$  \lim_{N \to \infty} \P\left( \sup_{r \leq R f(N) \ln N} \left\{Z_0(r)\right\} \leq \eps N \Big|Z_0(0) \leq \mathfrak{e}(\alpha,\mu,-\eps)\frac{N}{f(N)}\right)=1. $$
Moreover, as long as $Z_0(r) \leq \eps N$, the per capita birth rate of $Z_0$ may be bounded as follows:
$$ \frac{\alpha}{f(N)} (1-\eps) \leq b_0(Z_0(r),N) \leq \frac{\alpha}{f(N)},  $$
and since $-{\eps}/{f(N)} +  1/{2N} < 0$ for large $N$, then the per capita death rate of $Z_0$ satisfies
$$ \frac{\alpha}{f(N)} \left(\frac{\mu-\eps}{\alpha} + \frac{Z_0(r) f(N)}{2 \alpha N}\right) \leq d_0(Z_0(r),N) \leq \frac{\alpha}{f(N)} \left(\frac{\mu}{\alpha} + \frac{Z_0(r) f(N)}{2 \alpha N}\right).  $$
Applying again Lemma C.1 in \cite{champagnat2021stochastic}, we obtain that for any $R<\infty$,
\begin{multline*}  \lim_{N \to \infty} \P\Big(\frac{f(N)}{N}Z_0(r) \in \left[ \mathfrak{e}(\alpha,\mu,2\eps),\mathfrak{e}(\alpha,\mu,-2\eps)\right] \mbox{ for all }  r \leq R f(N) \ln N  \\ \Big| \frac{f(N)}{N} Z_0(0) \in \left[ \mathfrak{e}(\alpha,\mu,\eps),\mathfrak{e}(\alpha,\mu,-\eps)\right] \Big)=1. \end{multline*}
This proves the lemma for $k=0$ with $C_0(\eps)=6/\eps$ and $C_0=4\eps$.

Let us now take $g \in \N$ and assume that \eqref{containment_B12} holds true for $k=0,...,g-1$. Jointly for all these $k=0,...,g-1$ we can take a time frame (which may be as long as we want on the $f(N)\ln N$-time scale) on which $f(N)Z_k/N \in [\bar{n}_k - C_k\eps, \bar{n}_k + C_k\eps]$, and $Z_g \leq 4\alpha N/f(N) $. During this time interval, the  birth rate of the $Z_g$-population is larger than 
$$\frac{\mu}{f(N)}\left( \bar{n}_{g-1}- C_{g-1}\eps \right) \frac{N}{f(N)}+ Z_{g}\frac{\alpha}{f(N)}\left( 1- \frac{4\alpha}{f(N)} \right)  $$
and smaller than 
\vspace{-0.3cm}
$$\frac{\mu}{f(N)}\left( \bar{n}_{g-1}+ C_{g-1}\eps \right) \frac{N}{f(N)}+ Z_{g}\frac{\alpha}{f(N)}.  $$
Likewise, the death rate is larger than
$$  Z_g\left( \frac{\mu}{f(N)}+ \frac{Z_g-1}{2N} + \sum_{k=0}^{g-1}\left( \bar{n}_{k}- C_{k}\eps \right) \frac{N}{f(N)} \right)$$ 
\vspace{-0.2cm}
and smaller than
\vspace{-0.2cm}
$$  Z_g\left(\frac{\mu}{f(N)}+ \frac{Z_g-1}{2N} + \sum_{k=0}^{g-1}\left( \bar{n}_{k}+ C_{k}\eps \right) \frac{N}{f(N)}. \right)$$
The remaining part of the proof is the same as in the case $g=0$, again with an application of \cite[Lemma~C.1]{champagnat2021stochastic}.
\end{proof}
\begin{proof}[Proof of Lemma \ref{exp_rate}] 
Equation \eqref{QSD} (for some factor of $r$ in the exponent) follows from basic properties of quasi-stationary distributions (see for instance  Proposition 2 in \cite{meleard2012quasi}).

In order to show that the exponent has the claimed property, we will couple the process $Z_0$ with logistic birth and death processes, applying results from \cite{chazottes2016sharp}. To see this in detail, recall the definition of $b_0(.,N)$ and $d_0(.,N)$ in \eqref{Z0rates} and choose $\eps > 0$. Then for $N$ large enough and  any $n \in \N$, 
$$b_0(n,N) \leq \frac{\alpha n}{f(N)} \quad \text{and} \quad \frac{n}{f(N)} \left( \mu- \eps+ \frac{n f(N)}{2N}\right) \leq d_0(n,N) \leq \frac{n}{f(N)} \left( \mu+ \frac{n f(N)}{2N}\right), $$
and for $n \leq \eps N/\alpha$, 
\vspace{-0.2cm}
$$(\alpha-\eps)\frac{ n}{f(N)}\leq b_0(n,N).$$
Now let us consider two auxiliary birth and death processes $Z_0^{(+,N)}$ and $Z_0^{(-,N)}$, where 
\begin{equation*} 
Z_0^{(+,N)} \mbox{ has jump rates } \left\{ \begin{array}{lll}
n b_0(+,n,N):= \frac{\alpha n}{f(N)}  &\mbox{from } n \mbox{ to } n+1 ,\\
n d_0(+,n,N):= \frac{n}{f(N)} \left( \mu- \eps+ \frac{n f(N)}{2N}\right)&\mbox{from } n \mbox{ to } n-1,
\end{array}
\right .
\end{equation*}
\begin{equation*}
Z_0^{(-,N)} \mbox{ has jump rates } \left\{ \begin{array}{lll}
n b_0(-,n,N):= (\alpha-\eps)\frac{ n}{f(N)}  &\mbox{from } n \mbox{ to } n-1 ,\\
n d_0(-,n,N):= \frac{n}{f(N)} \left( \mu+ \frac{n f(N)}{2N}\right)\phantom{AA}&\mbox{from } n \mbox{ to } n-1.
\end{array}
\right .
\end{equation*}
In addition we consider the process $Z_0^{(R,-,N)}$ which has the same rates as the process $Z_0^{(-,N)}$ on the set $\{0,\ldots, \lfloor\eps N/\alpha\rfloor-1\}$ and is reflected below at $N_0 := \lfloor\eps N/\alpha\rfloor$, i.e. jumps from $N_0$ to $N_0-1$ at rate $N_0 d_0(-,N_0,N)$ but never jumps from $N_0$ to $N_0 +1$ (see Figure~\ref{fig:boat5}).
\begin{figure}[h!]
\begin{center}
  \includegraphics[width=.8\linewidth]{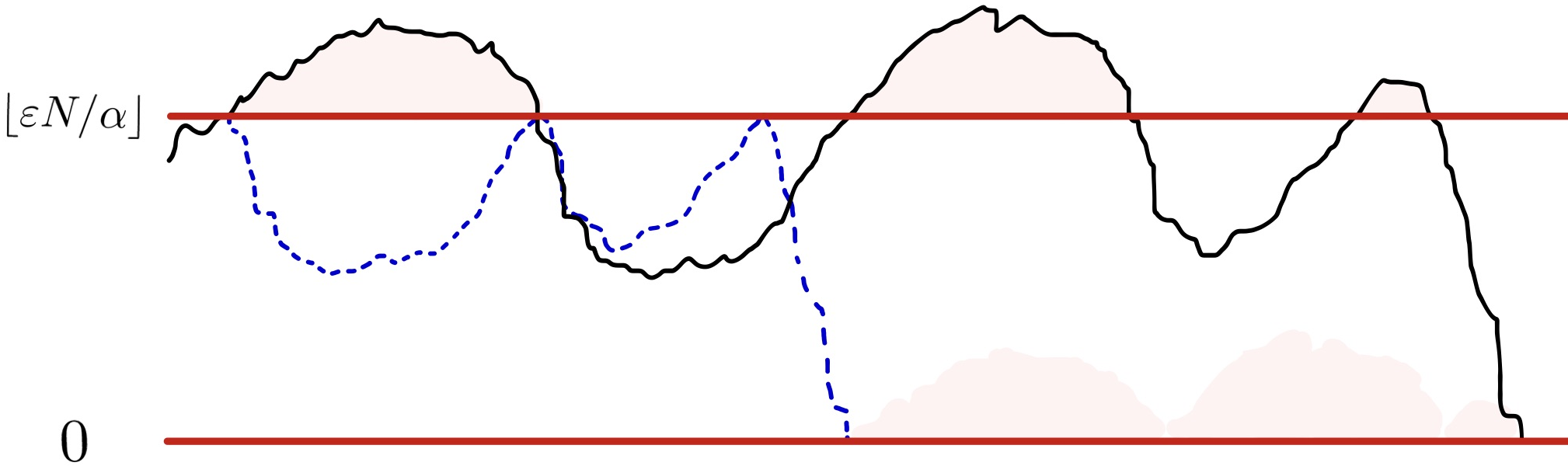}
  \end{center}
  \caption{This figure schematically displays $Z_0^{(-,N)}$ (path drawn solid) as well as  $Z_0^{(R,-,N)}$ (path drawn dashed), illustrating how the total length of the excursioins  of  $Z_0^{(-,N)}$ above level $\lfloor \eps N/\alpha\rfloor$ determines the difference of the times at which the two processes hit~$0$.}
  \label{fig:boat5}
\end{figure}

We can couple $Z_0$ and $Z_0^{(+,N)}$ such that for any $r \ge 0$, 
$$ Z_0(r) \leq Z_0^{(+,N)}(r). $$
Denoting by $\nu_N^{(\pm)}$ the quasi-stationary distributions of $Z_0^{(\pm,N)}$ and by  $\theta_N^{(\pm)}$ the real numbers such that
$$\P_{\nu^{(\pm)}_N} \left(Z_0^{(\pm,N)}(r)>0  \right) = e^{-r/f(N)\theta_N^{(\pm)}}, $$
we thus obtain an upper bound for $\theta_N$, namely
$ \theta_N \leq \theta_N^{(+)}$.
Indeed $Z_0$ (resp. $Z_0^{(+,N)}$) with an initial state of order $N/f(N)$ may be coupled with $Z_0$ (resp. $Z_0^{(+,N)}$) with initial distribution $\nu_N$ (resp. $\nu_N^{(+)}$) in such a way that they coincide after a time of order $f(N)\ln N$ (the proof is similar to that of Lemma \ref{lem_coupling1}).

We can also couple $Z_0$ and $Z_0^{(R,-,N)}$ such that for any $r$, 
\begin{equation} \label{coupleZ0below}
Z_0^{(R,-,N)}(r) \leq Z_0(r).
\end{equation}
Indeed, this relation is fulfilled if $Z_0(r) \ge  \lfloor\eps N/\alpha\rfloor$, since $Z_0^{(R,-,N)} \leq \lfloor\eps N/\alpha\rfloor$ by definition.  As long as $Z_0^{(R,-,N)} \leq \lfloor\eps N/\alpha\rfloor$, however, the process $Z_0^{(R,-,N)}$ has a smaller birth rate and a larger death rate than the process $Z_0$. 

The coupling \eqref{coupleZ0below} allows us to bound the mean extinction time of $Z_0$ by that of $Z_0^{(R,-,N)}$. In order to estimate the latter we will prove that the mean extinction times of the processes $Z_0^{(R,-,N)}$ and $Z_0^{(-,N)}$ are of the same order.
We will then apply results of \cite{meleard2012quasi} to get an equivalent of $\theta^{(-)}_N$, the mean extinction time of the process $Z_0^{(-,N)}$.
Let us consider a realization of the process $Z_0^{(-,N)}$, and denote by  $(\mathscr{T}^{up}_i,\mathscr{T}^{down}_i, i \in \N_0)$ the successive entrance and exit times of $[\lfloor\eps N/\alpha\rfloor+1,\infty)$ by the process $Z_0^{(-,N)}$, defined recursively as follows: 
\begin{eqnarray*}
\mathscr{T}^{up}_0&&=\quad \mathscr{T}^{down}_0 \, = \, 0,\\
 \mathscr{T}^{up}_i&&:=\quad  \inf \{ r \geq \mathscr{T}^{down}_{i-1}, Z_0^{(-,N)}(r)=\lfloor\eps N/\alpha\rfloor+1 \},  \\
\mathscr{T}^{down}_i&&:=\quad  \inf \{ r \geq \mathscr{T}^{up}_{i}, Z_0^{(-,N)}(r)=\lfloor\eps N/\alpha\rfloor \}, \qquad \qquad \qquad  i \in \N.
\end{eqnarray*}
Then if we ignore the excursions above $\lfloor\eps N/\alpha\rfloor$ and glue together the points beginning and ending these excursions (that is the point $\mathscr{T}^{up}_i$ and $\mathscr{T}^{down}_i$), we obtain a realization of the process $Z_0^{(R,-,N)}$, which is almost surely smaller than the process $Z_0^{(-,N)}$ at any time. To make this ``glueing of excursions'' formal, let us introduce the process $Z_0^{(C,-,N)}$ via
$$ Z_0^{(C,-,N)}(r):= Z_0^{(-,N)}\left(r- \sum_{j=1}^i \left(\mathscr{T}^{down}_j-\mathscr{T}^{up}_j\right)\right) $$
for $r$ in the random time interval
$$ \mathscr{T}^{up}_i-\sum_{j=1}^{i-1} \left(\mathscr{T}^{down}_j-\mathscr{T}^{up}_j\right) \leq r \leq \mathscr{T}^{up}_{i+1}- \sum_{j=1}^i \left(\mathscr{T}^{down}_j-\mathscr{T}^{up}_j\right). $$
This process obeys
$$ Z_0^{(C,-,N)} \stackrel{\mathcal{L}}{=}Z_0^{(R,-,N)} \quad \text{and} \quad Z_0^{(C,-,N)}(r) \leq Z_0^{(-,N)}(r), \quad \forall r \geq 0, \quad a.s. $$
Thus, if we prove that the mean time spent above $\lfloor\eps N/\alpha\rfloor$ by the process $Z_0^{(-,N)}$ is negligible with respect to its mean extinction time $\theta_N^{(-)}$, we may deduce that the mean extinction time of the process $Z_0^{(R,-,N)}$ is equivalent to $\theta_N^{(-)}$, which in turn entails that $\theta_N^{(-)}\leq \theta_N.$ 

As $\nu_N^{(-)}$ is a quasi-stationary distribution, we have for every $r\geq0$ (see for instance \cite{meleard2012quasi}),
\begin{align*} \nu_N^{(-)} (\lfloor\eps N/\alpha\rfloor ,\infty) &= \P_{\nu_N^{(-)}}\left( Z_0^{(-,N)}(r) \in (\lfloor\eps N/\alpha\rfloor ,\infty)\right)\left(\P_{\nu_N^{(-)}}\left(\mathcal{H}^{(N)}_0>r\right)\right)^{-1}\\
&= \P_{\nu_N^{(-)}}\left(Z_0^{(-,N)}(r) \in (\lfloor\eps N/\alpha\rfloor ,\infty)\right)e^{r/f(N)\theta_N^{(-)}}. \end{align*}
Hence, the expected time spent by the process $Z_0^{(-,N)}$ in the set $(\lfloor\eps N/\alpha\rfloor ,\infty)$ is
\begin{eqnarray*}
\E_{\nu_N^{(-)}} \left[ \int_0^\infty \mathbf{1}_{\{Z_0^{(-,N)}(r)>\lfloor\eps N/\alpha\rfloor \}}dr \right] =\int_0^\infty  \P_{\nu_N^{(-)}} \left( Z_0^{(-,N)}(r)>\lfloor\eps N/\alpha\rfloor  \right)dr \\
 = \nu_N^{(-)} (\lfloor\eps N/\alpha\rfloor ,\infty)\int_0^\infty e^{-r/f(N)\theta_N^{(-)}}dr
 = \nu_N^{(-)} (\lfloor\eps N/\alpha\rfloor ,\infty)f(N)\theta_N^{(-)},
\end{eqnarray*}
where we applied Fubini's Theorem.

Now from Theorem 3.7 in \cite{chazottes2016sharp}, we get that the total variation distance between $\nu_N^{(-)}$ and a Gaussian law centered at $2(\alpha-\mu-\eps)N/f(N)$ and with variance $N/((\alpha-\eps)f(N))$ is of order $\sqrt{f(N)/N}$. We deduce that
$$ \lim_{N \to \infty} \nu_N^{(-)} (\lfloor\eps N/\alpha\rfloor ,\infty)=0, $$
which allows us to conclude that
$\theta_N^{(-)}\leq \theta_N\leq \theta_N^{(+)}.$

A direct application of Remark 3.3 in \cite{chazottes2016sharp} yields that, as $N\to \infty$,
\begin{equation*}   \theta_N^{(-)} \sim f(N) (\alpha-\eps-\mu)^2 \sqrt{2\pi \mu\frac{f(N)}{2N}}\exp\left(2\left(\alpha -\eps-\mu + \mu \ln \frac{\mu}{\alpha-\eps}\right)\frac{N}{f(N)}\right) 
\end{equation*}
and
\begin{equation*}   \theta_N^{(+)} \sim f(N) (\alpha+\eps-\mu)^2 \sqrt{2\pi (\mu-\eps)\frac{f(N)}{2N}}\exp\left(2\left(\alpha +\eps-\mu + (\mu-\eps) \ln \frac{\mu-\eps}{\alpha}\right)\frac{N}{f(N)}\right).
\end{equation*}
 This concludes the proof of \eqref{order_time_ext}, and thus completes the proof of Lemma \ref{exp_rate}.
\end{proof}

\begin{proof}[Proof of Lemma \ref{lem_fast_ext}]
To prove {part a)} of the lemma, we will zoom on a small window before the extinction time of $Z_0=Z_0^{(N)}$, namely after the last hitting time of $1/\eps^2$. The strategy of the proof consists in showing that this time is short ({having} a duration of order $\ln 1/\eps$) and that {on the way to} extinction, $Z_0$ will feed the $Z_1$ population {by producing many individuals that carry one mutation}.

To begin with, we {consider} the process $Z_0$ conditioned to reach $0$ before ${\lfloor}1/\eps^2{\rfloor}$. {Its transition matrix  $\widehat{P}$ arises from that of the unconditioned $Z_0$ as the harmonic transform}
$$ \widehat{P}(i,i-1):= \frac{h(i-1)}{h(i)}\frac{\mu + \frac{f(N)(n-1)}{2N}}{\alpha \left(1- \frac{i}{N}\right)+\mu + \frac{f(N)(n-1)}{2N}}, $$
where 
\begin{equation} \label{harm}
h(i):= \P\left(H^{(N)}_0 < H^{(N)}_{\lfloor 1/\eps^2\rfloor}\Big|Z_0(0)=i\right), \quad { i=0,\ldots, \lfloor 1/\eps^2\rfloor},
\end{equation}
and $H^{(N)}_z$ is the first hitting time of $z$, cf. \eqref{hitting_Z0i}. The form of the harmonic functions for birth and death chains in terms of ratios of upward and downward jump rates is well known, cf. \cite[Example 1.3.4]{norris1998markov}. For \eqref{harm} this leads to the expression
$$ h(i)= 1- (1-h(1))\sum_{j=1}^i \prod_{k=1}^{j-1}\rho_k, \quad
\text{where} 
\quad \rho_k:= \frac{\mu + \frac{f(N)(k-1)}{2N}}{\alpha \left(1- \frac{k}{N}\right)}, \quad 1 \le k < N. $$
Hence we need to study the quantity
\begin{equation} \label{expr_harmo} \frac{h(i-1)}{h(i)}= \frac{1- (1-h(1))\sum_{j=1}^{i-1} \prod_{k=1}^{j-1}\rho_k}{1- (1-h(1))\sum_{j=1}^i \prod_{k=1}^{j-1}\rho_k}. \end{equation}
In fact a lower bound will be sufficient in our case. 
We notice that the expression in \eqref{expr_harmo} is non-decreasing if any of the $\rho_k$'s is increased, and that  $\rho_k \leq \mu/\alpha$ for any $k$. We thus get
$$ \frac{h(i-1)}{h(i)}\geq \frac{1- (1-h(1))\sum_{j=1}^{i-1} \prod_{k=1}^{j-1}\mu/\alpha}{1- (1-h(1))\sum_{j=1}^i \prod_{k=1}^{j-1}\mu/\alpha}=
 \frac{1- (1-h(1))\frac{\alpha}{\alpha-\mu}\left(1 - \left(\frac{\mu}{\alpha}\right)^{i-1}\right)}{1- (1-h(1))\frac{\alpha}{\alpha-\mu}\left(1 - \left(\frac{\mu}{\alpha}\right)^{i}\right)}.  $$
 The last step consists in finding an equivalent of $1-h(1)$ for large $N$. As for any $k \leq 1/\eps^2$, 
$$ \frac{\mu}{\alpha} \leq \rho_k \leq \frac{\mu + \frac{f(N)}{2N\eps^2}}{\alpha\left(1- \frac{1}{N\eps^2}\right)}, $$
we obtain from classical results on hitting probabilites for Galton-Watson process (see \cite[Theorem III.4.1]{athreya1972branching}, or \cite[Example 1.3.3]{norris1998markov} together with the fact that a continuous-time supercritical binary branching process is a time-changed drifted simple random walk and that the hitting probabilities are invariant under the time change) 
$$\left( 1 - \frac{\mu + \frac{f(N)}{2N\eps^2}}{\alpha\left(1- \frac{1}{N\eps^2}\right)} \right) \left( 1 - \left(\frac{\mu + \frac{f(N)}{2N\eps^2}}{\alpha\left(1- \frac{1}{N\eps^2}\right)}\right)^{1/\eps^2}\right)^{-1}  \leq 1-h(1) \leq \left( 1 - \frac{\mu}{\alpha} \right) \left( 1 - \left(\frac{\mu}{\alpha}\right)^{1/\eps^2} \right)^{-1}. $$
We thus get that
\vspace{-0.2cm}
$$ 1-h(1) = \frac{\alpha-\mu}{\alpha} + \delta(1/N)\left(\frac{\mu}{\alpha}\right)^{1/\eps^2},$$
where $\delta(1/N)\to 0$ as $N \to \infty$,
which entails that for any $1 \leq i \leq 1/\eps^2$, 
$$  \frac{h(i-1)}{h(i)} = \frac{\alpha}{\mu}, + \delta(1/N)$$
\vspace{-0.2cm}
and consequently
\begin{equation*} \widehat{P}(i,i-1) =  \frac{\alpha}{\alpha+\mu}+ \delta(1/N). \end{equation*}
In other words, the process $Z_0$ conditioned to reach $0$ before $1/\eps^2$ behaves as a subcritical branching process with individidual birth and death rates $\mu/f(N)$ and $\alpha/f(N)$. With a probability close to one when $\eps$ is small, it thus takes a time smaller than $4f(N)/(\alpha-\mu)\ln (1/\eps)$ to reach $0$.
We will now prove that during this time, on the way of extinction, a number of order $1/\eps^2$ of individuals of the $Z_1$ population  are produced by mutation from $Z_0$ individuals.

Let us denote by $\P^{(Z_0,Z_1)}_{(i,j)}((k,l))$ the probability for the process $(Z_0,Z_1)$ to perform its first jump to $(k,l)$ when starting in $(i,j)$. We then have
\vspace{-0.2cm}
\begin{align*} &\P^{(Z_0,Z_1)}_{(i,j)}\left((i-1,j+1)\Big|H^{(N)}_0 < H^{(N)}_{1/\eps^2}\right)
=\frac{\P^{(Z_0,Z_1)}_{(i,j)}\left((i-1,j+1),H^{(N)}_0 < H^{(N)}_{1/\eps^2}\right)}{\P^{(Z_0,Z_1)}_{(i,j)}\left(H^{(N)}_0 < H^{(N)}_{1/\eps^2}\right)}\\
=&\frac{\P^{(Z_0,Z_1)}_{(i,j)}\left((i-1,j+1)\right)}{\P^{(Z_0,Z_1)}_{(i,j)}\left(H^{(N)}_0 < H^{(N)}_{1/\eps^2}\right)}\P^{(Z_0,Z_1)}_{(i,j)}\left(H^{(N)}_0 < H^{(N)}_{1/\eps^2}\Big|(i-1,j+1)\right)\\
=&\frac{\P^{(Z_0,Z_1)}_{(i,j)}\left((i-1,j+1)\right)}{\P^{(Z_0,Z_1)}_{(i,j)}\left(H^{(N)}_0 < H^{(N)}_{1/\eps^2}\right)}\P^{(Z_0,Z_1)}_{(i-1,j+1)}\left(H^{(N)}_0 < H^{(N)}_{1/\eps^2}\right)
=\P^{(Z_0,Z_1)}_{(i,j)}\left((i-1,j+1)\right)\frac{h(i-1)}{h(i)}, \end{align*}
with $h(i)$ as in \eqref{harm}. The last equality is a consequence of the autonomy of the law of $Z_0$, cf. Remark \ref{cor_rates_01}. In the same way, we can prove that 
$$ \P^{(Z_0,Z_1)}_{(i,j)}\left((i-1,j)\Big|H^{(N)}_0 < H^{(N)}_{1/\eps^2}\right)
=\P^{(Z_0,Z_1)}_{(i,j)}\left((i-1,j)\right)\frac{h(i-1)}{h(i)}.$$
This entails that 
$$ \frac{\P^{(Z_0,Z_1)}_{(i,j)}\left((i-1,j+1)\Big|H^{(N)}_0 < H^{(N)}_{1/\eps^2}\right)}{\P^{(Z_0,Z_1)}_{(i,j)}\left((i-1,j)\Big|H^{(N)}_0 < H^{(N)}_{1/\eps^2}\right)}=\frac{\P^{(Z_0,Z_1)}_{(i,j)}\left((i-1,j+1)\right)}{\P^{(Z_0,Z_1)}_{(i,j)}\left((i-1,j)\right)}. $$
In other words, conditioning on the event that $Z_0$ is on its way to extinction does not modify the proportion of those deaths in the process $Z_0$ which lead to a creation of a new individual in the $Z_1$ population. Hence, if we denote by $\mathscr{M}(\eps)$ the number of mutations from $Z_0$ to $Z_1$ after the last visit of $Z_0$ to  $\lfloor 1/\eps^2 \rfloor$, we have
$$ \mathscr{M}(\eps) \geq \sum_{i=1}^{1/\eps^2}{\rm Be}^{(i)} \left( \frac{\mu}{\mu + \frac{f(N)(i-1)}{2N}} \right), $$ 
where the ${\rm Be}^{(i)}(\lambda_i)$'s are independent Bernoulli random variables with parameter $\lambda_i$.
Indeed, in the considered time interval  there are at least $1/\eps^2$ deaths of individuals in the $Z_0$~population, and the parameter of $\rm{Be}^{(i)}$ is the probability that a jump from $i$ to $i-1$ in the $Z_0$~population leads to the arrival of a new individual in the $Z_1$~population.
For $\eps$ small enough, the random variable $\mathscr{M}(\eps)$  is thus larger than $1/(2\eps^2)$ with a probability tending to~1 as $N\to \infty$.
The evolution of the $Z_1$ population after the last visit of $Z_0$ to $\lfloor 1/\eps^2 \rfloor$ can thus with high probability be coupled to the offspring of $\lfloor 1/(2\eps^2) \rfloor $ immigrants arriving in a time interval of length $4f(N)/(\alpha-\mu)\ln (1/\eps)$, with their offspring suffering an individual death rate by competition from $Z_0$ individuals which is smaller than $1/N\eps^2$, as well as a reduction of their individual birth rate (compared to $\alpha$) that is smaller than 
$\alpha / (f(N)\eps^2 N)$.  This proves part a) of Lemma~\ref{lem_fast_ext}. 

To prove part b),  in view of Lemmata \ref{lem_coupling1} and \ref{exp_rate} it suffices to restrict to the event $\{\underline k_N > 0\}$. We can then work iteratively along the extinction times of the $Z_0$, \ldots $Z_{\underline k_N-1}$ populations, applying part a) by induction. 
\end{proof}

%%%%%%%%%%%%%%%%%%%%%%%%%%%%%%%%%%%%%%%%%%%%%%%%%%%%%%%%%%%%%%%%%%%
%%                                                               %%
%% No need for \maketitle.                                       %%
%%                                                               %%
%%%%%%%%%%%%%%%%%%%%%%%%%%%%%%%%%%%%%%%%%%%%%%%%%%%%%%%%%%%%%%%%%%%

%%%%%%%%%%%%%%%%%%%%%%%%%%%%%%%%%%%%%%%%%%%%%%%%%%%%%%%%%%%%%%%%%%%
%%                                                               %%
%% Please replace what follows by the body of your article       %%
%% (up to the bibliography):                                     %%
%%                                                               %%
%%%%%%%%%%%%%%%%%%%%%%%%%%%%%%%%%%%%%%%%%%%%%%%%%%%%%%%%%%%%%%%%%%%

%%%%%%%%%%%%%%%%%%%%%%%%%%%%%%%%%%%%%%%%%%%%%%%%%%%%%%%%%%%%%%%%%%%
%%                                                               %%
%% Use the two commands below for producing your bibliography    %%
%% with bibtex, then comment again the commands and include the  %%
%% content of the .bbl file in this file below the commands.     %%
%%                                                               %%
%%%%%%%%%%%%%%%%%%%%%%%%%%%%%%%%%%%%%%%%%%%%%%%%%%%%%%%%%%%%%%%%%%%

\bibliographystyle{amsplain}
%\bibliography{bibliofancyratchet}
% add below the content of your .bbl file produced by bibtex.

\providecommand{\bysame}{\leavevmode\hbox to3em{\hrulefill}\thinspace}
\providecommand{\MR}{\relax\ifhmode\unskip\space\fi MR }
\providecommand{\MRhref}[2]{%
  \href{http://www.ams.org/mathscinet-getitem?mr=#1}{#2}
}
\providecommand{\href}[2]{#2}

%%%%%%%%%%%%%%%%%%%%%%%%%%%%%%%%%%%%%%%%%%%%%%%%%%%%%%%%%%%%%%%%%%%
%%                                                               %%
%% You may add acknowledgments (optional).                       %%
%%                                                               %%
%%%%%%%%%%%%%%%%%%%%%%%%%%%%%%%%%%%%%%%%%%%%%%%%%%%%%%%%%%%%%%%%%%%
\begin{acks}
We are grateful to Ralph Neininger for pointing us to \cite{AB05} in connection with Proposition~\ref{pi_eqals_p}, see Remark \ref{geome}. We also thank Antar Bandyopadhyay for related discussions, Mariana Gomez for showing us the link between our model and tournament selection, Jan Lukas Igelbrink and Peter Pfaffelhuber for helpful comments, and an anonymous reviewer for a very careful reading and a number of constructive queries and suggestions that helped to improve the presentation.  Our sincere thanks go to the Mathematical Research Institute Oberwofach where core parts of this project were carried out in a joint stay as Oberwolfach research fellows in summer 2022. AGC and AW acknowledge the hospitality of INRAE and Institut Fourier in October 2021; during this joint visit (funded by the French program Investissement
d'avenir), substantial preparatory studies were performed. AGC acknowledges the hospitality of Goethe University during his visit to Frankfurt in September 2022.
\end{acks}

%%%%%%%%%%%%%%%%%%%%%%%%%%%%%%%%%%%%%%%%%%%%%%%%%%%%%%%%%%%%%%%%%%%
%%                                                               %%
%% You have reached the end of your document.                    %%
%%                                                               %%
%%%%%%%%%%%%%%%%%%%%%%%%%%%%%%%%%%%%%%%%%%%%%%%%%%%%%%%%%%%%%%%%%%%

\end{document}